\documentclass[12pt,a4paper]{amsart}

\usepackage{amssymb}
\usepackage{amsthm}
\usepackage{amsmath}
\usepackage{graphicx}
\usepackage{enumerate}
\usepackage{xcolor}
\usepackage{tikz}
\usepackage{graphbox}
\usepackage{orcidlink}
\usepackage{mathrsfs}  
\usepackage{hyperref}
\hypersetup{colorlinks=true,allcolors=blue}
\usepackage[margin=2.5cm]{geometry}

\newtheorem{theorem}{Theorem}[section]
\newtheorem{proposition}[theorem]{Proposition}

\newtheorem{lemma}[theorem]{Lemma}

\theoremstyle{definition}
\newtheorem{example}[theorem]{Example}
\newtheorem{definition}[theorem]{Definition}

\newcommand{\II}{\mathbb{I}}
\newcommand{\RR}{\mathbb{R}}
\newcommand{\dk}{\delta_k}
\newcommand{\DK}{\Delta_k}
\newcommand{\SK}{\mathcal{S}_k}
\newcommand{\PK}{\mathcal{P}_k}
\newcommand{\PKhat}{\widehat{\mathcal{P}}_k}
\newcommand{\Chat}{\widehat{C}}
\newcommand{\Xhat}{\widehat{X}}
\newcommand{\Yhat}{\widehat{Y}}
\newcommand{\PP}{\mathbb{P}}
\newcommand{\EE}{\mathbb{E}}
\newcommand{\VV}{\mathbb{V}}
\newcommand{\Cov}{\mathbb{C}\mathrm{ov}}
\newcommand{\BK}{\mathcal{B}_k}
\newcommand{\XX}{{\bf x}}

\newcommand{\aaa}{\boldsymbol{\alpha}}
\newcommand{\bbb}{\boldsymbol{\beta}}
\newcommand{\ccc}{\boldsymbol{\gamma}}
\newcommand{\be}{\boldsymbol{1}}
\newcommand{\hell}{\hat{\ell}}
\newcommand{\ba}{\boldsymbol{a}}
\newcommand{\Dir}{\textup{Dir}}
\newcommand{\vol}{\mathrm{V}}

\title{Random discrete copulas}

\author[D. Kokol Bukovšek]{Damjana Kokol Bukovšek \orcidlink{0000-0002-0098-6784} }
\address{University of Ljubljana, School of Economics and Business, and Institute of Mathematics, Physics and Mechanics, Ljubljana, Slovenia}
\email{Damjana.Kokol.Bukovsek@ef.uni-lj.si}

\author[B. Mojškerc]{Blaž Mojškerc \orcidlink{0000-0001-8096-355X} }
\address{University of Ljubljana, School of Economics and Business, and Institute of Mathematics, Physics and Mechanics, Ljubljana, Slovenia}
\email{Blaz.Mojskerc@ef.uni-lj.si}

\author[N. Stopar]{Nik Stopar \orcidlink{0000-0002-0004-4957} } 
\address{University of Ljubljana, Faculty of Civil and Geodetic Engineering, University of Ljubljana, Faculty of Mathematics and Physics, and Institute of Mathematics, Physics and Mechanics, Ljubljana, Slovenia}
\email{Nik.Stopar@fgg.uni-lj.si}


\begin{document}

\begin{abstract}
We introduce the notion of a bivariate random discrete copula on an equidistant mesh and explore its stochastic properties.
A random discrete copula is a discrete random field, hence, its value at a given point on the mesh is a random variable.
We determine the distribution of this random variable and calculate its expected value and variance.
We also consider bilinear extension of a random discrete copula to a random field over the whole unit square. 
\end{abstract}

\maketitle

\section{Introduction}

Ever since Sklar's seminal paper~\cite{Skl59} on multivariate distributions in 1959, copulas have become one of the main statistical tools for modeling dependence. They serve as a link between the joint cumulative distribution functions of random variables and their univariate marginal distributions. 
In particular, they allow one to model continuous random vectors by separately modeling the marginal distributions and the dependence structure between the components.

In practice, a copula joining a pair of random variables may not be known.  
There are several approaches to addressing such situations. 
Traditional methods involve estimating the unknown copula either parametrically or non-parametrically, using available empirical data. 
In the parametric approach, a practitioner chooses an appropriate parametric family of copulas, based on past experience, and estimates the unknown parameters using the data, see \cite{FreVal98, CraCra08, KimSilSil07}. A non-parametric approach, on the other hand, does not require any prior knowledge about the copula.  Instead, the unknown copula is estimated directly using empirical copulas, see \cite{Deh79, CapFouGen97, GenNesRem14}. 

An alternative to the traditional methods is to use imprecise models, for example an imprecise copula, see \cite{MonMirPelVic15, OmlSto20}. 
This model replaces a single unknown copula with a family of feasible copulas. In particular, an imprecise copula is a set of all copulas lying between two quasi-copulas $A$ and $B$. 
On one hand, this allows one to use only functions $A$ and $B$ in calculations rather than the whole family of copulas, see for example \cite{PiEsGoPu26}. 
On the other hand, this means that the information on the intermediate copulas is essentially ignored.  If one could incorporate the information about the copulas between $A$ and $B$ into the model, say, by adding additional structure to the set of all such copulas, it would improve the model's performance, since it could be used to make better statistical estimations with higher confidence levels.

One way, which is largely unexplored in the literature, would be to introduce randomness when modeling the copula, in particular, by treating the unknown copula as a random element of some imprecise copula.
The notion of a random copula has not yet been developed, in fact, it is not entirely clear what a random copula should be. For this reason we believe that it is essential to first investigate random copulas in the discrete setting, in order to gain an insight into their behavior. The discrete case is easier to treat, since combinatorial techniques can be used. This will enable us to later tackle random copulas defined on the whole unit square, using the results of the discrete case as guidelines. Hence,
in this paper we initiate the investigation of random copulas in the discrete setting. In particular, we define  a bivariate random discrete copula on an equidistant mesh as a discrete random field, whose realizations are discrete copulas. 
We develop two specific models, the first one for the class of discrete copulas induced by permutations, and the second one for the class of all discrete copulas on an equidistant mesh.

We take advantage of the fact that 
bivariate discrete copulas defined on an equidistant mesh are in a bijective correspondence with bistochastic matrices, see \cite{KolMesMordSem06}. 
In particular, a bistochastic matrix represents the mass distribution of some bivariate discrete copula. Random bistochastic matrices have been studied in paper \cite{ChaDiaSly10}. 
The set of bistochastic matrices of order $k$ is a convex polytope $\mathscr{B}_k$ in a real vector space, called the Birkhoff polytope, see \cite{Zie95}. 
The authors of \cite{ChaDiaSly10} choose a random element from $\mathscr{B}_k$ uniformly with respect to the Lebesgue measure and investigate its stochastic properties. 
In paper \cite{CapSomBruZyc09}, the authors present an algorithm to generate random bistochastic matrices, based on an algorithm of Sinkhorn, \cite{Sin64}, and evaluate the probability measure induced on $\mathscr{B}_k$.

In this paper, we take a different approach by using the property that every bistochastic matrix is a convex combination of permutation matrices. 
The set of all possible $m$-tuples of coefficients is a simplex $\sigma^{m-1}$ in a high-dimensional real space, so we introduce a random discrete copula by choosing the $m$-tuple of coefficients randomly and uniformly with respect to the Lebesgue measure on $\sigma^{m-1}$. 
This approach is easier to handle since the volume of $\sigma^{m-1}$ is simple to evaluate, while the formulas for the volume of $\mathscr{B}_k$ are complicated and known in closed form only for small values of $k$, as demonstrated in \cite{BecPix03}.

As indicated by the above, developing any model for a random discrete copula requires making a specific choice of the underlying distribution on a chosen class of copulas. In our models we used the following principles for choosing the underlying distribution:
\begin{enumerate}[(P1)]
    \item \emph{Simplicity:} We want our choices to be as natural and standard  as possible and avoid the need for numerical simulations at the very beginning. We also want to be able to sample a random discrete copula directly. 
    \item \emph{Computability:} We want the model to have explicit closed form formulas for expected value and variance of the value of a random discrete copula at every point in the copula domain. This property was chosen since the obtained formulas will later serve as a guideline for the continuous case and as a point of comparison with alternative models.
    \item \emph{Broadness:} A random copula can be used to model an unknown dependence structure between random variables. We want our model to be applicable in situations when we have no prior knowledge about the dependence structure.
\end{enumerate}

In our first model of a random discrete copula induced by a permutation we use the uniform distribution on the set of all permutations on $k$ elements. This choice is the simplest one and is most often used in the discrete setting when no additional information is available. Given that it is also computable, it seems to be the most obvious choice. 
The same choice was made also in \cite{ErGoNe08}, where the symmetry of a random discrete copula is studied. In our second model of a random discrete copula on an equidistant mesh we use the uniform distribution on the simplex of coefficients in a convex combination, as explained above. 
This choice might not be the most natural one -- a more obvious choice would be the uniform distribution on the Birkhoff polytope itself, see \cite{ChaDiaSly10}. However, uniform distribution on the Birkhoff polytope violates computability, since even the formulas for the volume of the polytope are not in a closed form. Furthermore, the exact uniform sampling from the Birkhoff polytope is an open problem, although approximate sampling algorithms do exist, \cite{CapSomBruZyc09}. Since computability is one of our main goals, we prefer to use the uniform distribution on the simplex of coefficients. It is also possible to sample from this distribution directly.

The study of discrete copulas induced by permutations and general discrete copulas on an equidistant mesh was initiated in \cite{MaySunTor05} and \cite{KolMesMordSem06}. Matrix representations of such copulas were also explored there, however, randomness was not considered.
Copulas induced by permutations are also considered in \cite{ErGoNe08}, where the distribution of a specific measure of asymmetry of such copulas is determined and used as a test statistic for the symmetry of a copula.
Discrete copulas induced by permutations are tightly connected to empirical copulas, i.e., copula-like objects obtained from random samples, which are used in goodness-of-fit tests for copulas, see \cite{GeReBe09}.
In \cite{GenNesRem14} the asymptotic behavior of multivariate empirical copulas when the sample size tends to infinity is studied.  As already mentioned, random discrete copulas on an equidistant mesh are considered in \cite{ChaDiaSly10}, where the Lebesgue measure on the Birkhoff polytope is taken.
An approximate sampling algorithm for this setting is considered in \cite{CapSomBruZyc09}. The idea of using Dirichlet distribution on the simplex $\sigma^m$ of coefficients of a convex combination was previously applied in \cite{MelPet95} in Bayesian inference for contingency tables, a topic closely related to discrete copulas. 
The geometry of discrete copulas on a general mesh and their connection to various polytopes is discussed in \cite{PeSoUh19}, while imprecise discrete copulas, i.e., intervals of discrete copulas on an equidistant mesh, are considered in \cite{KoPe25}.
The authors of \cite{Gee20,KoMa24} discuss the use of copulas for pairs of discrete random variables and suggest a different approach for modeling the dependence structure in this setting.

One of the motivations for this study are possible applications. Random discrete copulas could be used in simulations, when we have very few information on the dependence structure. 
In some situations, the type of dependence could even change unpredictably over time (compare \cite{FeWe12,YaCaHaLi22}), so past data may not reflect well the future joint behavior of random variables.
In this case, random copulas would be a useful tool for generating artificial data samples.
Furthermore, one could test the performance of various statistical methods on pairs of random variables connected by randomly chosen copulas, rather than on pairs connected by specific copulas chosen by the tester (see also Section~\ref{se:simulation}). 
This would give a more objective assessment of the performance of the method in question. 

The main contribution of this paper is the conceptualization of randomness for copulas in the discrete setting. We introduce two models and for each of them we explicitly compute the expected value and the variance of the value of the random discrete copula at each point in its domain. Furthermore, we extend these two models to checkerboard copulas.
The paper is structured as follows. In Section~\ref{se:prelim} we give the definition of a bivariate (discrete) copula and recall some properties of the Dirichlet distribution. 
In Section~\ref{se:permutation} we introduce a random discrete copula induced by a permutation and determine the distribution, expected value and variance of its value at an arbitrary point. 
In Section~\ref{se:general} we define a random discrete copula on an equidistant mesh and study its properties, extending the results from Section~\ref{se:permutation}. We illustrate the distribution with an example.
Section~\ref{se:extension} is devoted to bilinear extension of a random discrete copula to a random field over the whole unit square.
In Section~\ref{se:simulation} we present a numerical simulation as an application of random discrete copulas to concordance measures.
We finish with some concluding remarks in Section~\ref{se:concluding}.

\section{Preliminaries} \label{se:prelim}

We start by recalling the definition of a copula. Denote with $\II=[0,1]$ the unit interval.

\begin{definition}
A function  $C\colon\II^2\to\II$ is called a bivariate \emph{copula}, if
\begin{enumerate}[(i)]
    \item $C$ is \emph{grounded}, i.e., $C(u, 0)=C(0, v)=0$ for all $u, v \in \II$,
    \item $C$ has \emph{uniform marginals}, i.e., $C(u, 1)=u$ and $C(1, v)=v$ for all $u, v \in \II$, and
    \item $C$ is \emph{$2$-increasing}, i.e., for any rectangle $R=[x,x+u]\times[y,y+v] \subseteq\II^2$ with $u,v\ge 0$, the $C$-volume of $R$ is nonnegative, 
$$\vol_C ([x,x+u]\times[y,y+v]) := C(x+u,y+v) - C(x+u,y) - C(x,y+v) + C(x,y) \ge 0.$$
\end{enumerate}
\end{definition}
Recall that the lower and upper bound for bivariate copulas are the Fr\'{e}chet-Hoeffding bounds
$$W(u,v)=\max \{0,u+v-1\} \qquad \text{and} \qquad M(u,v)=\min \{u,v \}.$$

Let $k \ge 2$ be an integer. We denote $[k]=\{1,2, \ldots, k\}$ and $[k]_0=[k] \cup\{0\}$. With $\dk= \{ 0,\tfrac1k, \tfrac2k, \ldots, \tfrac{k-1}{k}, 1 \}$ we denote an equidistant partition of the unit interval $\II$, and with $\DK= \dk \times \dk$ an equidistant mesh in $\II^2$.

\begin{definition}
A function  $C\colon\DK\to\II$ is called a bivariate \emph{discrete copula}, if
\begin{enumerate}[(i)]
    \item $C$ is \emph{grounded}, i.e., $C(\tfrac{i}{k}, 0)=C(0, \tfrac{j}{k})=0$ for all $i,j \in [k]_0$,
    \item $C$ has \emph{uniform marginals}, i.e., $C(\tfrac{i}{k}, 1)=\tfrac{i}{k}$ and $C(1, \tfrac{j}{k})=\tfrac{j}{k}$ for all $i, j \in [k]_0$, and
    \item $C$ is \emph{$2$-increasing}, i.e., for any $i,j \in [k]$ it holds that
\begin{equation*}
    V_C([\tfrac{i-1}{k},\tfrac{i}{k}]\times[\tfrac{j-1}{k},\tfrac{j}{k}]) = C(\tfrac{i}{k},\tfrac{j}{k}) + C(\tfrac{i-1}{k},\tfrac{j-1}{k}) - C(\tfrac{i-1}{k},\tfrac{j}{k}) - C(\tfrac{i}{k},\tfrac{j-1}{k}) \ge 0.
\end{equation*}
\end{enumerate}
\end{definition}

Any discrete copula $C$ can be extended to a copula $\widehat{C}$ by using bilinear functions. This (piecewise) bilinear extension is often referred to as a checkerboard copula, see \cite{MikTay10,LinWanZhaZha25}. For any $(u,v) \in \II^2$ let 
$$u = \tfrac{i}{k} + \tfrac{t}{k} \qquad \text{and} \qquad v = \tfrac{j}{k} + \tfrac{s}{k},$$
where $i=\lfloor uk \rfloor \in [k]_0$, $j=\lfloor vk \rfloor \in [k]_0$, $t = uk -i \in [0,1)$, and $s = vk -j \in [0,1)$. The copula $\widehat{C}$ is then given as
$$\widehat{C}(u,v) = (1-t)(1-s)C(\tfrac{i}{k},\tfrac{j}{k}) + t(1-s)C(\tfrac{i+1}{k},\tfrac{j}{k}) + (1-t)sC(\tfrac{i}{k},\tfrac{j+1}{k}) +tsC(\tfrac{i+1}{k},\tfrac{j+1}{k}).$$

We also recall some properties of the Dirichlet distribution (see \cite{JonKot72,KotBalJoh00}).
With ${\bf x}$ we denote the $m$-tuple ${\bf x}=(x_1,\ldots,x_m)$ and with $$\sigma^{m-1}= \{ \XX \in \RR^m \mid x_i \ge 0 \text{ for all $i$ and } \sum_{i=1}^m x_i =1 \}$$ the standard simplex of dimension $m-1$.
Furthermore, with $$\sigma_*^m= \{ \XX \in \RR^m \mid x_i \ge 0 \text{ for all $i$ and } \sum_{i=1}^m x_i \le 1 \}$$ we denote the corner simplex of dimension $m$. Note that $\sigma^{m-1}$ is a facet of $\sigma_*^m$ lying opposite to the origin.
With $\be_A$ we denote the characteristic function of the set $A$, i.e., $\be_A(\XX)=1$ if $\XX \in A$ and $\be_A(\XX)=0$ otherwise.

Let $\mathbf{X}=(X_1,X_2,\ldots,X_m)$ be a random vector with support $\sigma^{m-1}$. We say that $\mathbf{X}$ has a Dirichlet distribution $\Dir(\ba)$ with parameters $\ba=(a_1,a_2,\ldots,a_m) \in (0,\infty)^m$ if $X_{m}=1-X_1-X_2-\ldots-X_{m-1}$ and $(X_1,X_2,\ldots,X_{m-1})$ is a continuous random vector with density 
\begin{equation}\label{eq:density}
    p(x_1,x_2,\ldots,x_{m-1})=\frac{\Gamma(a_1+a_2+\ldots+a_m)}{\Gamma(a_1)\Gamma(a_2)\cdots\Gamma(a_m)} x_1^{a_1-1} x_2^{a_2-1} \cdots x_{m}^{a_{m}-1} \be_{\sigma_*^{m-1}}(x_1,x_2,\ldots,x_{m-1}),
\end{equation}
where $x_m=1-x_1-x_2-\ldots-x_{m-1}$.
In particular, if $\ba=(1,1,\ldots,1) \in (0,\infty)^m$ then $\Dir(\ba)$ is the uniform distribution on $\sigma^{m-1}$.

Let $\mathbf{X} \sim \Dir(\ba)$, where $\ba \in (0,\infty)^m$ and $a_0=a_1+a_2+\ldots+a_m$. Then
\begin{equation}\label{eq:EVC}
\begin{aligned}
    \EE[X_i] &=\frac{a_i}{a_0}, \\
    \VV[X_i] &=\frac{a_i(a_0-a_i)}{a_0^2(a_0+1)}, \quad \text{and}\\
    \Cov[X_i,X_j] &=-\frac{a_ia_j}{a_0^2(a_0+1)} \quad \text{for all $i \ne j$.}
\end{aligned}
\end{equation}
The Dirichlet distribution has the \emph{aggregation} property, namely, if $\mathbf{X} \sim \Dir(a_1,a_2,\ldots,a_m)$ and $\mathbf{Y}=(Y_1,Y_2,\ldots,Y_k)$, where $Y_i=\sum_{j \in I_i} X_j$ and $I_1,I_2,\ldots,I_k$ is a partition of the index set $[m]=\{1,2,\ldots,m\}$, then $\mathbf{Y} \sim \Dir(b_1,b_2,\ldots,b_k)$, where $b_i=\sum_{j \in I_i} a_j$.

\section{Random discrete copula induced by a permutation} \label{se:permutation}

In this section we introduce a random discrete copula induced by a permutation as a discrete random field over an equidistant mesh, and determine its basic properties. For an introduction to random fields, we refer the reader to \cite{AdlTay07}.

Let $k \ge 2$ be an integer. With $\SK$ we denote the set of all permutations of the set $[k]$, $\SK= \{ \pi \colon [k] \to [k] \mid  \pi \text{ bijective} \}$.
For any permutation $\pi \in \SK$, let $C_\pi$ be a discrete copula with mass $\frac{1}{k}$ in each square $[\frac{i-1}{k},\frac{i}{k}] \times [\frac{\pi(i)-1}{k},\frac{\pi(i)}{k}]$, $i \in [k]$, and mass $0$ in all remaining squares determined by the mesh $\DK$. We denote by
\begin{equation}\label{eq:P_k}
    \PK=\{ C_\pi \colon \DK \to \II \mid \pi \in \SK \}
\end{equation}
the set of all such discrete copulas. Note that $\displaystyle C_\pi$ can also be constructed as the restriction of a shuffle of $M$ to $\DK$, i.e., $\displaystyle C_\pi = M(k,\dk, \pi, 1)\big|_{\DK}$, (see \cite[\S 3.2.3]{Nel06} or \cite{KoBuMo22} for the definition of a shuffle). An example of copula $C_\pi$ with $\pi =  \left(\begin{smallmatrix}
       1 & 2 & 3 & 4 & 5 & 6 & 7 & 8 & 9 & 10 \\
       3 & 7 & 2 & 10 & 5 & 9 & 1 & 4 & 6 & 8
   \end{smallmatrix}\right)$ is shown in Figure~\ref{fig:Cpi}. The squares containing mass $\frac{1}{10}$ are shaded gray.

\begin{definition}
We choose a discrete copula $C_\pi\in \PK$ uniformly at random. For every $(u,v) \in \DK$ we denote by
$$X_k(u,v)=C_\pi(u,v)$$
the value of $C_\pi$ at the point $(u,v)$. This makes $X_k(u,v)$ a discrete random variable with range in $\dk$ and $X_k$ a discrete random field over $\DK$. We call $X_k$ a \emph{random discrete copula induced by a permutation}.      
\end{definition}

We will now determine the distribution of random variable $X_k(u,v)$ for every $(u,v) \in \DK$, i.e., compute the probabilities $\PP[X_k(u,v)= t]$ for $t \in \dk$. 
Since $(u,v) \in \DK$ and $t \in \dk$, there exist $i,j, \ell\in [k]_0$ such that 
$$u=\tfrac{i}{k},  \quad v=\tfrac{j}{k}, \quad \text{and} \quad t= \tfrac{\ell}{k}.$$ 
We fix this notation throughout this section.

\begin{proposition}\label{pr:3.2}
    For any $(u,v) \in \DK$ and $t \in \dk$, we have 
    \begin{equation} \label{eq:pxk=t}
        \PP[X_k(u,v)=t] = \frac{(ku)! \, (kv)! \, (k(1-u))! \, (k(1-v))!}{k! \, (k(u-t))! \, (k(v-t))! \, (k(1-u-v+t))! \, (kt)!}.
    \end{equation} 
\end{proposition}

\begin{proof}
Let 
\begin{equation}\label{eq:L_ijl}
     L_{ij\ell} = \{\pi \in \SK \mid C_\pi(\tfrac{i}{k},\tfrac{j}{k}) = \tfrac{\ell}{k}\}.
\end{equation}
If $\pi \in L_{ij\ell}$ then 
\begin{align*}
|\{m \in [k] \mid m\le i, \pi(m) \le j\}| &= \ell, \\
|\{m \in [k] \mid m\le i, \pi(m) \ge j+1\}| &= i-\ell, \\
|\{m \in [k] \mid m\ge i+1, \pi(m) \le j\}| &= j-\ell, \quad\text{and} \\
|\{m \in [k] \mid m\ge i+1, \pi(m) \ge j+1\}| &= k+\ell-i-j. 
\end{align*}
To obtain a permutation $\pi \in L_{ij\ell}$, we first choose a subset $A \subseteq [i]$ containing $\ell$ elements and a subset $B \subseteq [j]$ containing $\ell$ elements and choose any bijective mapping from $A$ to $B$. Next we choose a subset $C \subseteq [k]\setminus [j]$ containing $i-\ell$ elements and choose any bijective mapping from $[i]\setminus A$ to $C$. Next we choose a subset $D \subseteq [k]\setminus [i]$ containing $j-\ell$ elements and choose any bijective mapping from $D$ to $[j]\setminus B$. Finally, we choose any bijective mapping from $[k]\setminus([i]\cup D)$ to $[k]\setminus([j]\cup C)$. Example~\ref{ex:permutation}
illustrates this process. 

Counting these choices, it follows that 
\begin{align}
    |L_{ij\ell}|
    &= \binom{i}{\ell} \, \binom{j}{\ell} \, \ell! \, \binom{k-j}{i-\ell} \, (i-\ell)! \, \binom{k-i}{j-\ell} \, (j-\ell)! \,(k+\ell-i-j)! \nonumber\\
    &= \frac{i! \,j! \, \ell!\, (k-j)! \,(j-\ell)! \,(k-i)! \,(j-\ell)! \, (k+\ell-i-j)!}{\ell! \, (i-\ell)! \, \ell! \, (j-\ell)! \, (i-\ell)! \, (k+\ell-i-j)! \, (j-\ell)! \, (k+\ell-i-j)!} \label{eq:L_ijl_size}\\
    &= \frac{i! \, j! \, (k-i)! \, (k-j)!}{(i-\ell)! \, (j-\ell)! \, (k+\ell-i-j)! \, \ell!}. \nonumber
\end{align}
We now have
\begin{align*}
    \PP[X_k(u,v)=t] &= \PP[X_k(\tfrac{i}{k},\tfrac{j}{k})= \tfrac{\ell}{k}] = \frac{1}{k!} |L_{ij\ell}| \\
    &= \frac{i! \, j! \, (k-i)! \, (k-j)!}{k! \, (i-\ell)! \, (j-\ell)! \, (k+\ell-i-j)! \, \ell!} \\ 
    &= \frac{(ku)! \, (kv)! \, (k(1-u))! \, (k(1-v))!}{k! \, (k(u-t))! \, (k(v-t))! \, (k(1-u-v+t))! \, (kt)!},
\end{align*}
which finishes the proof. 
\end{proof}

Note that the probability $\PP[X_k(u,v)=t] = \PP[X_k(\tfrac{i}{k},\tfrac{j}{k})= \tfrac{\ell}{k}]$ corresponds to the probability that a randomly chosen permutation $\pi$ satisfies $\pi(r)\le j$ for precisely $\ell$ numbers $r \in \{1,2,\ldots,i\}$.

\begin{example} \label{ex:permutation}
   Let $k=10$, $i=4$, $j=5$, and $\ell=2$. We illustrate the process of choosing a permutation $\pi \in L_{452}$ from the proof of Proposition~\ref{pr:3.2}. We first choose the sets 
   \begin{align*}
       A &= \{1,3\} \subseteq [4], \\
       B &= \{2,3\} \subseteq [5], \\
       C &= \{7,10\} \subseteq [10]\setminus[5], \text{ and} \\
       D &= \{5,7,8\} \subseteq [10]\setminus[4]. 
   \end{align*}
   Then we choose bijective maps 
   $$\{1,3\} \to \{2,3\}, \quad \{2,4\} \to \{7,10\}, \quad \{5,7,8\} \to \{1, 4, 5\}, \quad \text{ and }  \quad \{6, 9, 10\} \to \{6, 8, 9\},$$ 
   and combine them into a permutation 
   $$ \pi = \begin{pmatrix}
       1 & 2 & 3 & 4 & 5 & 6 & 7 & 8 & 9 & 10 \\
       3 & 7 & 2 & 10 & 5 & 9 & 1 & 4 & 6 & 8
   \end{pmatrix}.$$
   The mass distribution of copula $C_\pi$ is depicted in Figure~\ref{fig:Cpi}. The rows and columns corresponding to sets $A,B,C,D$  are shadowed red, blue, green and yellow, respectively. 
\end{example}

\begin{figure}
    \centering
    \includegraphics[width=0.4\linewidth]{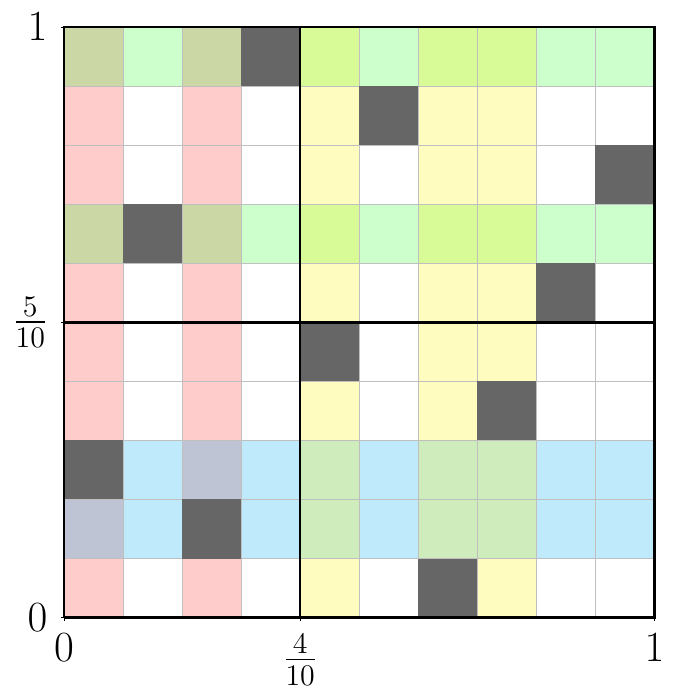}
    \caption{Mass distribution of discrete copula $C_\pi$ from Example~\ref{ex:permutation}. }
    \label{fig:Cpi}
\end{figure}

Note that if $t \in \dk$ and $\max \{0, u+v-1 \} \le t \le \min \{ u,v \}$, all the factorials in \eqref{eq:pxk=t} have nonnegative integer arguments, which makes $\PP[X_k(u,v)=t]\ne 0$. If $t < u+v-1$ or $t > \min \{ u,v \}$ then $\PP[X_k(u,v)=t]= 0$ since such $t$ violates the Fr\'{e}chet-Hoeffding bounds for copulas. Hence, we adopt the convention  $\tfrac{1}{m!} = 0$ if $m$ is a negative integer. This is consistent with the extension of the factorial to the Gamma function $m! = \Gamma(m+1)$ and the fact that the Gamma function has simple poles at negative integers. 

We now compute the expected value and variance of random variable $X_k(u,v)$. In the proof we will use Vandermonde's identity
\begin{equation} \label{eq:vsota-binomskih}
   \sum_{s=0}^{\min\{a,c \}} \binom{a}{s} \, \binom{b}{c-s}  = \sum_{s=0}^{a} \binom{a}{s} \, \binom{b}{c-s} = \sum_{s=0}^c \binom{a}{s} \, \binom{b}{c-s} = \binom{a+b}{c}.
\end{equation}

\begin{theorem} \label{th:exk}
    Let $k\ge 2$. For any $(u,v) \in \DK$ we have 
    $$ \EE[X_k(u,v)] = uv \qquad \text{ and } \qquad \VV[X_k(u,v)] = \frac{uv(1-u)(1-v)}{k-1}.$$
\end{theorem}

\begin{proof}
We first compute the expected value of variable $X_k(u,v)$. If $i=0$ then $X_k(u,v) = 0$ and the claim holds. So suppose that $i\ge 1$. Then
\begin{align*}
    \EE[X_k(u,v)] &= \sum_{\ell=\max \{ 0,i+j-k\}}^{\min \{i, j\}} \frac{\ell}{k}\cdot  \frac{i! \, j! \, (k-i)! \, (k-j)!}{k! \, (i-\ell)! \, (j-\ell)! \, (k+\ell-i-j)! \, \ell!} \\
    &= \frac{j! \, i \, (k-j)!}{k! \, k} \sum_{\ell=\max \{ 0,i+j-k\}}^{\min \{i, j\}} \frac{\ell\, (i-1)! \, (k-i)!}{\ell! \, (i-\ell)! \, (j-\ell)! \, (k+\ell-i-j)!}.
\end{align*}
Note that the  expression in the last sum is equal to $0$ for $\ell = 0$ and $\ell < i+j-k$, due to the convention that $\tfrac{1}{m!}=0$ for negative integers $m$. So this sum can be taken from $\ell =1$ on and can be simplified as follows
\begin{align*}
    \sum_{\ell=\max \{ 0,i+j-k\}}^{\min \{i, j\}} &\frac{\ell\, (i-1)! \, (k-i)!}{\ell! \, (i-\ell)! \, (j-\ell)! \, (k+\ell-i-j)!}\\
    &= \sum_{\ell=1}^{\min \{i, j\}} \frac{(i-1)! \, (k-i)!}{(\ell-1)! \, (i-\ell)! \, (j-\ell)! \, (k+\ell-i-j)!} \\
    &= \sum_{\ell=1}^{\min \{i, j\}} \binom{i-1}{\ell-1} \, \binom{k-i}{j-\ell} \\
    &= \sum_{s=0}^{\min \{i, j\}-1} \binom{i-1}{s} \, \binom{k-i}{j-1-s} = \binom{k-1}{j-1},
\end{align*}
where in the last step we used formula \eqref{eq:vsota-binomskih}.
But then
\begin{align*}
    \EE[X_k(u,v)] &= \frac{j! \, i \, (k-j)!}{k! \, k} \binom{k-1}{j-1} = \frac{j! \, i \, (k-j)!}{k! \, k} \cdot \frac{(k-1)!}{(j-1)! \, (k-j)!} = \frac{ij}{k^2} = uv.
\end{align*}

To determine $\VV[X_k(u,v)]$, we first prove that  $$\EE\big[X_k(u,v)^2\big] = \frac{ij}{k^3} + \frac{i(i-1)j(j-1)}{k^3 (k-1)}.$$
If $i=0$ or $j=0$ then $X_k(u,v) = 0$ and the formula holds. So suppose that $i,j \ge 1$. Using similar sum manipulations as above, and reindexing the sum with $s=\ell-1$, we obtain
\begin{align*}
    \EE\big[X_k(u,v)^2\big] &= \sum_{\ell=\max \{0,  i+j-k \}}^{\min \{i,j \}} \frac{\ell^2}{k^2}\cdot \frac{i! \, j! \, (k-i)! \, (k-j)!}{k! \, \ell! \, (j-\ell)! \, (k+\ell-i-j)!} \\
    &= \frac{i \, j! \, (k-j)!}{k! \, k^2 } \sum_{\ell=1}^{\min \{ i,j \}} \ell \, \binom{i-1}{\ell-1} \, \binom{k-i}{j-\ell} \\
    &= \frac{i}{k^2 \, \binom{k}{j}} \sum_{s=0}^{\min \{ i,j \}-1} (s+1) \, \binom{i-1}{s} \, \binom{k-i}{j-s-1} \\
    &= \frac{i}{k^2 \, \binom{k}{j}} \sum_{s=0}^{\min \{ i,j \}-1} \binom{i-1}{s} \, \binom{k-i}{j-s-1} \\
    & \qquad \qquad+ \frac{i}{k^2 \, \binom{k}{j}} \sum_{s=1}^{\min \{ i,j \}-1} \frac{s \, (i-1)!}{s! \, (i-s-1)!} \binom{k-i}{j-s-1}.
\end{align*}
The first sum equals $\binom{k-1}{j-1}$ by equation \eqref{eq:vsota-binomskih}. If $i>1$ the second sum is
\begin{align*}
\sum_{s=1}^{\min \{ i,j \}-1} \frac{s \, (i-1)!}{s! \, (i-s-1)!} \binom{k-i}{j-s-1} &= (i-1)\sum_{s=1}^{\min \{ i,j \}-1} \frac{(i-2)!}{(s-1)! \, (i-s-1)!} \binom{k-i}{j-s-1} \\
&= (i-1)\sum_{s=1}^{\min \{i,j \}-1 } \binom{i-2}{s-1} \, \binom{k-i}{j-s-1} \\
&= (i-1)\binom{k-2}{j-2}
\end{align*}
by equation \eqref{eq:vsota-binomskih}. The above formula also holds if $i=1$, since the sum is empty in this case. We thus obtain
\begin{align}
    \EE\big[X_k(u,v)^2\big] &= \frac{i}{k^2 \, \binom{k}{j}} \binom{k-1}{j-1} + \frac{i(i-1)}{k^2 \binom{k}{j}} \binom{k-2}{j-2} \label{eq:E(X_k^2)}\\
    &= \frac{ij}{k^3} + \frac{i(i-1) \, j!\, (k-j)! \, (k-2)!}{k^2 \, k! \, (j-2)! \, (k-j)!} = \frac{ij}{k^3} + \frac{i(i-1)j(j-1)}{k^3 (k-1)}, \nonumber
\end{align}
as claimed. Finally, we determine the variance of $X_k(u,v)$
\begin{align*}
    \VV[X_k(u,v)] &= \EE\big[X_k(u,v)^2\big] - \EE\big[X_k(u,v)\big]^2 = \frac{ij}{k^3} + \frac{i(i-1)j(j-1)}{k^3(k-1)} - \frac{i^2 j^2}{k^4} \\
    &= \frac{ij(k^2-ki-kj+ij)}{k^4(k-1)}= \frac{i}{k}\cdot \frac{j}{k}\cdot \frac{k-i}{k}\cdot \frac{k-j}{k}\cdot \frac{1}{k-1} = \frac{uv(1-u)(1-v)}{k-1},
\end{align*}
which finishes the proof. 
\end{proof}

As anticipated, the expected value of $X_k$ is the product copula. Furthermore, the variance is the largest in the middle of the unit square and drops to zero at its boundary. This is also expected since $M(u,v)-W(u,v)=\min \{u,v,1-u,1-v\}$.

\medskip

Notice that for any $(u,v) \in \II^2$ the $C$-volume of a rectangle $[0,u]\times[0,v]$ equals the value $C(u,v)$, i.e. $\vol_C([0,u]\times[0,v])=C(u,v)$.
Now we can rewrite \eqref{eq:pxk=t} in terms of $C$-volume of rectangle $[0,u]\times[0,v]$ for $u=\tfrac{i}{k}$ and $v=\tfrac{j}{k}$, so that for $t=\tfrac{\ell}{k}$ we have
$$\PP[X_k(u,v)= t] = \PP[\vol_{X_k}([0,u]\times[0,v])=t].$$
Here, we interpret $\vol_{X_k}$ as the volume with respect to a particular realization of $X_k$, i.e., as the composition of the volume function with random discrete copula $X_k$.
To generalize this result, we want to replace the rather specific rectangle $[0,u]\times[0,v]$ in the formula with a general one, $[x,x+u]\times[y,y+v]$  with $x,y,x+u,y+v \in \dk$, and compute
$\PP\big[\vol_{X_k}([x,x+u]\times[y,y+v])=t\big]$.

\begin{proposition}
For any $x,y,x+u,y+v,t \in \dk$ with $u,v \ge 0$ we have
$$\PP\big[\vol_{X_k}([x,x+u]\times[y,y+v])=t\big] = \PP[X_k(u,v)=t].$$  
Hence,
$\PP\big[\vol_{X_k}([x,x+u]\times[y,y+v])=t\big]$ is independent of the position of the rectangle and depends only on its dimensions.
\end{proposition}

\begin{proof}
Let $x=\tfrac{a}{k}$, $y=\tfrac{b}{k}$, and extend the definition of the set $L_{ij\ell}$ in \eqref{eq:L_ijl} for a general rectangle:
\begin{equation*}
     L_{R,t} = \{\pi \in \SK \mid \vol_{X_k}([x,x+u]\times[y,y+v])= t\}.
\end{equation*}
If $\pi \in L_{R,t}$, then
\begin{align*}
|\{m \in [k] \mid a+1 \le m \le a+i, b +1 \le \pi(m) \le b+j \}| &= \ell, \\
|\{m \in [k] \mid a+1 \le m \le a+i, \pi(m) \le b \,\, \text{or} \,\, \pi(m) \ge b+j+1 \}| &= i-\ell, \\
|\{m \in [k] \mid m \le a \,\, \text{or} \,\, m \ge a+i+1,  b+1 \le \pi(m) \le b+j \}| &= j-\ell, \quad\text{and} \\
|\{m \in [k] \mid m \le a \,\, \text{or} \,\, m \ge a+i+1, \pi(m) \le b \,\, \text{or} \,\, \pi(m) \ge b+j+1\}| &= k+\ell-i-j. 
\end{align*}
To obtain a permutation $\pi \in L_{R,t}$, we first choose a subset $A \subseteq \{a+1,\ldots,a+i\}$ containing $\ell$ elements, a subset $B \subseteq \{b+1,\ldots,b+j\}$ also containing $\ell$ elements, and choose any bijective mapping from $A$ to $B$. Next we choose a subset $C \subseteq [k]\setminus \{b+1,\ldots,b+j\}$ containing $i-\ell$ elements and choose any bijective mapping from $\{a+1,\ldots,a+i\}\setminus A$ to $C$. Then we choose a subset $D \subseteq [k]\setminus \{a+1,\ldots,a+i\}$ containing $j-\ell$ elements and choose any bijective mapping from $D$ to $\{b+1,\ldots,b+j\}\setminus B$. Finally, we choose any bijective mapping from $[k]\setminus(\{a+1,\ldots,a+i\}\cup D)$ to $[k]\setminus(\{b+1,\ldots,b+j\}\cup C)$.
Counting the choices, it follows that
\begin{align*}
    |L_{R,t}|
    &= \binom{i}{\ell} \, \binom{j}{\ell} \, \ell! \, \binom{k-j}{i-\ell} \, (i-\ell)! \, \binom{k-i}{j-\ell} \, (j-\ell)! \,(k+\ell-i-j)!=|L_{ij\ell}|,
\end{align*}
so we have
$$\PP\big[\vol_{X_k}([x,x+u]\times[y,y+v])=t\big] = \PP\big[\vol_{X_k}([0,u]\times[0,v])=t\big]=\PP[X_k(u,v)=t],$$
which finishes the proof.
\end{proof}

\section{Random discrete copula on an equidistant mesh} \label{se:general}

So far we have determined the distribution, the expected value and the variance of a value of a random discrete copula induced by a permutation at each point of $\DK$. Now we would like to extend this notion to discrete copulas defined on $\DK$. 

With $\BK$ we denote the set of all discrete copulas defined on $\DK$. These are in a one-to-one correspondence with the Birkhoff polytope $\mathscr{B}_k$, i.e., the set of all $k \times k$ bistochastic matrices, see \cite[Proposition~2]{KolMesMordSem06}. It follows from \cite[Proposition~4]{KolMesMordSem06} that $\BK = \textrm{co}(\PK)$, the convex hull of $\PK$ (see equation~\eqref{eq:P_k}), i.e., any discrete copula on $\DK$ can be written as a finite convex combination of discrete copulas from $\PK$.

Recall that $\SK$ denotes the set of all permutations of $[k]$. As announced in the introduction, we will consider convex combinations of elements from $\PK$ and assume that the $k!$-tuples $\aaa=(\alpha_\pi)_{\pi \in \SK}$ of their coefficients are distributed uniformly on simplex $\sigma^{k!-1}$ with respect to the Lebesgue measure.

\begin{definition}\label{def:Y_k}
We choose $\aaa \in \sigma^{k!-1}$ uniformly at random with respect to the Lebesgue measure.  For every $(u,v) \in \DK$ we define
$$Y_k(u,v)= \sum_{\pi\in \SK} \alpha_\pi \, C_{\pi}(u,v).$$
For every $(u,v) \in \DK$ from the interior of $\II^2$, this makes $Y_k(u,v)$ a continuous random variable whose range is a subset of $\II$. Thus, $Y_k$ is a discrete random field over $\DK$. We call $Y_k$ a \emph{random discrete copula}.
\end{definition}

 Since the uniform distribution on $\sigma^{k!-1}$ is a Dirichlet distribution with all parameters equal to $1$, we have by equation~\eqref{eq:EVC} 
\begin{equation}\label{eq:EVC_uni}
\begin{aligned}
    \EE[\alpha_\pi] &=\frac{1}{k!} \quad\text{for all } \pi \in \SK, \\
    \VV[\alpha_\pi] &=\frac{k!-1}{(k!)^2(k!+1)} \quad\text{for all } \pi \in \SK, \quad \text{and}\\
    \Cov[\alpha_\pi,\alpha_\tau] &=-\frac{1}{(k!)^2(k!+1)} \quad \text{for all } \pi, \tau \in \SK, \pi \ne \tau.
\end{aligned}
\end{equation}

In order to prove the main theorem of this section, we need the following lemma.
\begin{lemma} \label{lem:cov_YY}
    Let $k \geq 2$. For any $(u,v), (u',v')\in \DK$ we have 
    $$\Cov[Y_k(u,v),Y_k(u',v')] =\frac{1}{k!+1}\Cov[X_k(u,v),X_k(u',v')].$$
\end{lemma}

\begin{proof}
Let $(u,v), (u',v')\in \DK$. Using equation~\eqref{eq:EVC_uni}, we compute
    \begin{align*}
    \Cov[Y_k(u,v),&Y_k(u',v')] =\sum_{\pi \in \SK} \VV[\alpha_\pi] C_{\pi}(u,v)C_{\pi}(u',v')+\sum_{\substack{\pi,\tau \in \SK\\ \pi \ne \tau}} \Cov[\alpha_\pi,\alpha_\tau] C_{\pi}(u,v)C_{\tau}(u',v') \\
    &=\frac{k!-1}{(k!)^2(k!+1)} \sum_{\pi \in \SK} C_{\pi}(u,v)C_{\pi}(u',v')-\frac{1}{(k!)^2(k!+1)} \sum_{\substack{\pi,\tau \in \SK\\ \pi \ne \tau}} C_{\pi}(u,v)C_{\tau}(u',v') \\
    &=\frac{1}{k!(k!+1)} \sum_{\pi \in \SK} C_{\pi}(u,v)C_{\pi}(u',v')-\frac{1}{(k!)^2(k!+1)} \sum_{\pi,\tau \in \SK} C_{\pi}(u,v)C_{\tau}(u',v') \\
    &=\frac{1}{k!+1}\left(\frac{1}{k!}\sum_{\pi \in \SK} C_{\pi}(u,v)C_{\pi}(u',v')-\Big(\frac{1}{k!}\sum_{\pi \in \SK} C_{\pi}(u,v)\Big)\Big(\frac{1}{k!}\sum_{\tau \in \SK} C_{\tau}(u',v')\Big)\right) \\
    &=\frac{1}{k!+1}\big(\EE[X_k(u,v)X_k(u',v')]-\EE[X_k(u,v)]\, \EE[X_k(u',v')]\big) \\
    &=\frac{1}{k!+1}\Cov[X_k(u,v),X_k(u',v')],
\end{align*}
which finishes the proof.
\end{proof}

In the next theorem we compute the expected value and variance of random variable $Y_k(u,v)$.
\begin{theorem} \label{th:eyk}
    Let $k\ge 2$. For any $(u,v) \in \DK$ we have 
    $$\EE[Y_k(u,v)] =\EE[X_k(u,v)]= uv$$
    and
    $$\VV[Y_k(u,v)] =\frac{1}{k!+1}\VV[X_k(u,v)] = \frac{uv(1-u)(1-v)}{(k!+1)(k-1)}.$$
\end{theorem}

\begin{proof}
Let $(u,v) \in \DK$. Using equation~\eqref{eq:EVC_uni} and Theorem~\ref{th:exk}, we compute
$$\EE[Y_k(u,v)]= \sum_{\pi \in \SK} \EE[\alpha_\pi] C_{\pi}(u,v) = \frac{1}{k!} \sum_{\pi \in \SK} C_{\pi}(u,v) = \EE[X_k(u,v)] = uv.$$
Taking $(u',v')=(u,v)$ in Lemma~\ref{lem:cov_YY} and using Theorem~\ref{th:exk}, we obtain
$$\VV[Y_k(u,v)]=\frac{1}{k!+1}\VV[X_k(u,v)]=\frac{uv(1-u)(1-v)}{(k!+1)(k-1)}.$$
This concludes the proof.
\end{proof}

Next, we derive a formula for the cumulative distribution function of $Y_k(u,v)$.
\begin{proposition}
For every $(u,v)=(\tfrac{i}{k},\tfrac{j}{k}) \in \DK$ let
$$N=\min\{i,j\}-\max\{0,i+j-k\},$$
and for every $\ell \in \bigl\{\max\{0,i+j-k\},\ldots,\min\{i,j\} \bigr\}$ let
$$\hell=\ell-\max\{0,i+j-k\}, \quad \text{and} \quad n_{\hell}=|L_{ij\ell}|.$$
Then for every $w \in \II$ we have
    \begin{equation*}
    \PP[Y_k(u,v) \le w] =\frac{(k!-1)!}{\prod_{\hell=0}^N (n_{\hell}-1)!} \int_{\substack{t_1,t_2,\ldots,t_N \ge 0\hspace{0.67cm}\\t_1+t_2+\ldots+t_N \le 1 \hspace{1.11cm}\\ 1t_1+2t_2+\ldots+Nt_N \le k(w-W(u,v))}} t_0^{n_{0}-1} t_1^{n_{1}-1}\cdots t_N^{n_{N}-1} dt_1 \ldots dt_N,
\end{equation*}
where $t_0=1-t_1-t_2-\ldots-t_N$.
\end{proposition}
\begin{proof}
We first collect the terms of the sum for $Y_k$, see Definition~\ref{def:Y_k}, that have the same value $C_{\pi}(u,v)=\frac{\ell}{k}$ for some $\ell \in \bigl\{\max\{0,i+j-k\},\ldots,\min\{i,j\}\bigr\}$.
By equation~\eqref{eq:L_ijl}, we have
\begin{equation*}
    Y_k(u,v) =\sum_{\pi \in \SK} \alpha_\pi C_{\pi}(u,v) =\sum_{\ell=\max\{0,i+j-k\}}^{\min\{i,j\}} \Big(\sum_{\pi \in L_{ij\ell}}\alpha_\pi\Big) \frac{\ell}{k}=\sum_{\ell=\max\{0,i+j-k\}}^{\min\{i,j\}} \beta_{ij\ell} \cdot \frac{\ell}{k},
\end{equation*}
where $\beta_{ij\ell}=\sum_{\pi \in L_{ij\ell}}\alpha_\pi$. Next, we re-index the above sum, so that the summation index $\hell=\ell-\max\{0,i+j-k\}$ runs from $0$ to $N=\min\{i,j\}-\max\{0,i+j-k\}$. Since $\aaa=(\alpha_\pi)_{\pi \in \SK}$ is distributed uniformly on $\sigma^{k!-1}$, i.e., according to the Dirichlet distribution with all parameters equal to $1$, the aggregation property of the Dirichlet distribution implies that $\bbb=(\beta_{ij\ell})_{\ell=\max\{0,i+j-k\}}^{\min\{i,j\}}$ is distributed on $\sigma^N$ according to the Dirichlet distribution with parameters $(|L_{ij\ell}|)_{\ell=\max\{0,i+j-k\}}^{\min\{i,j\}}$, whose sum is $k!$.
We denote $\gamma_{ij\hell}=\beta_{ij\ell}$, so that $\ccc=(\gamma_{ij\hell})_{\hell=0}^N=\bbb$, and let $n_{\hell}=|L_{ij\ell}|$, $\hell=0,1,\ldots,N$.
Then
\begin{equation*}
    Y_k(u,v) =\sum_{\hell=0}^N \gamma_{ij\hell} \cdot \frac{\hell+\max\{0,i+j-k\}}{k}=W(u,v)+\sum_{\hell=1}^N \gamma_{ij\hell} \cdot \frac{\hell}{k},
\end{equation*}
since $\frac{\max\{0,i+j-k\}}{k}=\max\{0,\frac{i}{k}+\frac{j}{k}-1\}=W(u,v)$ and $\sum_{\hell=0}^N \gamma_{ij\hell}=1$.
Hence,
\begin{equation*}
    \PP[Y_k(u,v) \le w] =\PP\big[\sum_{\hell=0}^N \gamma_{ij\hell} \cdot \frac{\hell}{k} \le w-W(u,v)\big]=\PP\big[\sum_{\hell=1}^N \hell \cdot \gamma_{ij\hell} \le k(w-W(u,v))\big].
\end{equation*}
Taking into account that $\ccc$ is distributed according to the Dirichlet distribution with parameters $(n_{\hell})_{\hell=0}^N$ and using its probability density function (see equation~\eqref{eq:density}), we obtain
\begin{equation*}
    \PP[Y_k(u,v) \le w] =\frac{(k!-1)!}{\prod_{\hell=0}^N (n_{\hell}-1)!} \int_{\substack{t_1,t_2,\ldots,t_N \ge 0\hspace{0.67cm}\\t_1+t_2+\ldots+t_N \le 1 \hspace{1.11cm}\\ 1t_1+2t_2+\ldots+Nt_N \le k(w-W(u,v))}} t_0^{n_{0}-1} t_1^{n_{1}-1}\cdots t_N^{n_{N}-1} dt_1 \ldots dt_N,
\end{equation*}
where $t_0=1-t_1-t_2-\ldots-t_N$. This finishes the proof.
\end{proof}

Note that the domain of integration is the simplex $\sigma_*^N$ intersected by the half-space $1t_1+2t_2+\ldots+Nt_N \le k(w-W(u,v))$. This implies that the distribution function of $Y_k(u,v)$ is obtained by integrating the density of a Dirichlet distribution over a truncated simplex, where the truncation depends on the point $(u,v)$. We remark that $N$ and all $n_{\hell}$ also depend on $(u,v)$.

\medskip

We demonstrate the above with an example.

\begin{example} Let $k=4$ and $u=v=\frac12$, so that $i=j=2$. Then $N=2$ and $(n_0,n_1,n_2)=(|L_{220}|,|L_{221}|,|L_{222}|)=(4,16,4)$. The distribution function of the random variable $Y_4(\frac12,\frac12)$ equals
\begin{align*}
    \PP[Y_4(\tfrac12,\tfrac12) &\le w] =\frac{23!}{3! \cdot 15! \cdot 3!} \int_{\substack{\hspace{0.06cm} t_1, t_2 \ge 0\\ t_1+t_2 \le 1 \hspace{0.08cm}\\ t_1+2t_2\le 4w}} (1-t_1-t_2)^{3} t_1^{15} t_2^{3} ~dt_1 dt_2\\
    &=\begin{cases}
        0, & w<0\\
        \tfrac{1}{16} (4 w)^{20} (-72352 w^3+62744 w^2-18216 w+1771), & 0 \le w<\tfrac14,\\
        1-\tfrac{1}{16} (2-4w)^{20} (72352 w^3-45784 w^2+9736 w-695), & \tfrac14 \le w<\tfrac12,\\
        1, & \tfrac12 \le w,
    \end{cases}
\end{align*}
and its density equals
$$p(w)=\begin{cases}
    23(4w)^{19} (-18088 w^3+15004 w^2-4158 w+385), & 0 \le w<\tfrac14,\\
    23(2-4w)^{19} (18088 w^3-12128 w^2+2720 w-204), & \tfrac14 \le w<\tfrac12,\\
    0, & \text{otherwise}.
\end{cases}$$
Furthermore, $\EE[Y_4(\tfrac12,\tfrac12)]=\tfrac14$ and $\VV[Y_4(\tfrac12,\tfrac12)]=\tfrac{1}{1200}$.
Figure~\ref{fig:k4_graphs} depicts the graphs of the densities of the random variables $Y_4(u,v)$ for all $(u,v) \in \Delta_4$ with $u,v \notin \{0,1\}$.
Figure~\ref{fig:k4_density} depicts the heat-maps of the densities of variables $Y_4(u,v)$, $(u,v) \in \Delta_4$, on their respective ranges. The heat-maps are color-coded with respect to Figure~\ref{fig:k4_graphs}. Darker tone of the respective color represents higher density. The gray surface with the mesh lines represents the graph of the product copula.

\begin{figure}[ht!]
    \centering
    \includegraphics[width=\linewidth]{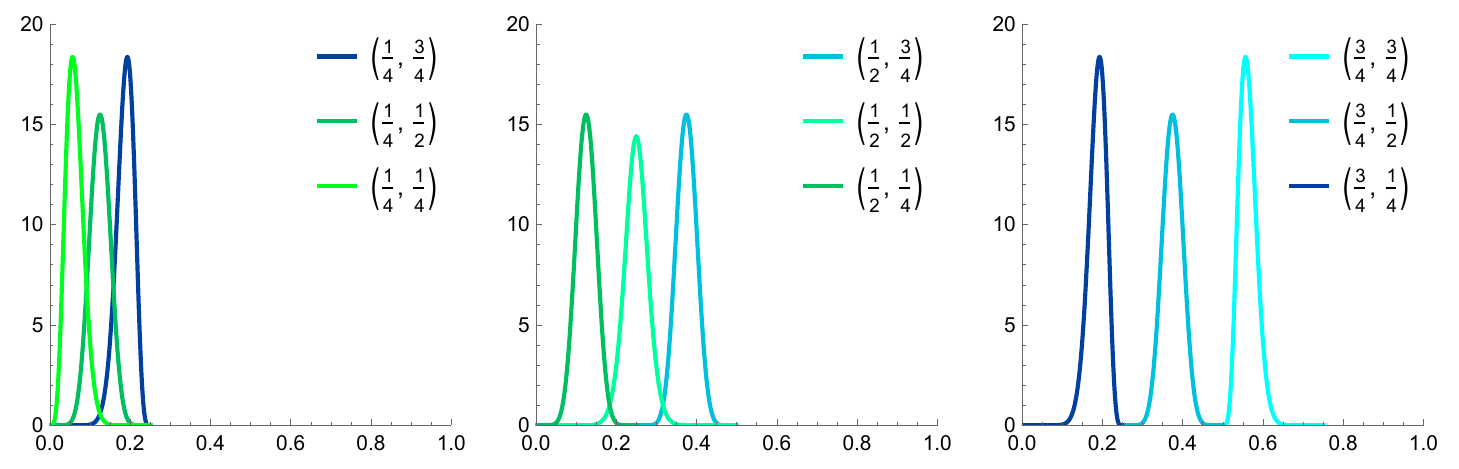}
    \caption{The densities of the random variables $Y_4(u,v)$ for all $(u,v) \in \Delta_4$ with $u,v \notin \{0,1\}$.}
    \label{fig:k4_graphs}
\end{figure}
\begin{figure}[ht!]
    \centering
    \includegraphics[width=0.5\linewidth]{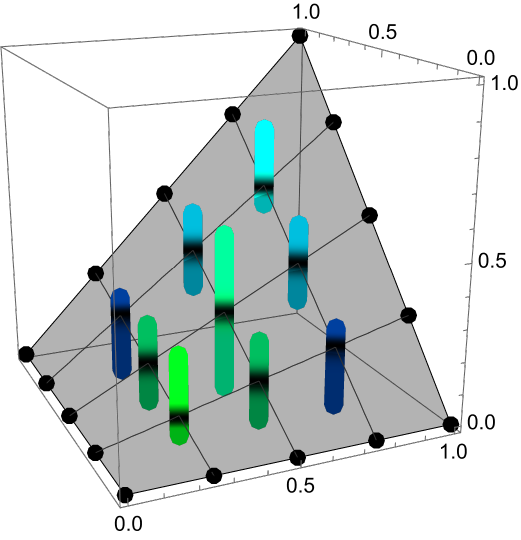}
    \caption{Heat-maps of the densities of random variables $Y_4(u,v)$ for all $(u,v) \in \Delta_4$.}
    \label{fig:k4_density}
\end{figure}

\end{example}

We remark that in practical applications the sampling of $Y_k$ is computationally demanding due to the need to sample $k!$ coefficients, which is feasible for small values of $k$ only. This issue will be addressed in our future work. We believe that the model could be modified so that the number $k!$ is brought down to $k^2$. Nevertheless, the present model is applicable for the analysis of small random samples.

\section{Bilinear extensions of random discrete copulas} \label{se:extension}

In practice, when the copula $C$ of a continuous random vector $(X,Y)$ is estimated, a finite data set is used. In particular, an empirically obtained sample $\{(x_i,y_i)\}_{i=1}^n$ of the random vector $(X,Y)$ is transformed into pseudo-observations $u_i=\tfrac{r_i}{n+1}$, $v_i=\tfrac{s_i}{n+1}$, where $r_i$ (respectively $s_i$) is the rank of observation $x_i$ (respectively $y_i$) among $x_1, \ldots, x_n$ (respectively $y_1,\ldots,y_n$).
This naturally gives rise to a discrete copula
$$\widetilde{C}(\tfrac{i}{n},\tfrac{j}{n})= \tfrac{1}{n} \sum_{\ell=1}^n \be_{ \{(u,v) \in \II^2 \, \mid \, u_\ell \leq u, \, v_\ell \leq v\}}(\tfrac{i}{n},\tfrac{j}{n}),$$
which is a discrete estimate for $C$ (see \cite[\S 5.6]{Nel06}), and is, in fact, a discrete copula induced by a permutation.
The bigger the data set is, the finer the domain of the discrete copula we get. General discrete copulas can be obtained using aggregation. In particular, we may take a divisor $k$ of the sample size $n$ and restrict $\widetilde{C}$ to $\DK$, see e.g. \cite{JuGrTr21}.

The so obtained discrete copula can then be extended to a continuous copula, which is an approximation of the original copula $C$. 
Often, the original copula is expected to be absolutely continuous. In this case, the simplest option is to use piecewise bilinear extensions (also called checkerboard extensions), see e.g. \cite{MikTay10,JuGrTr21}.
In fact, this corresponds to the choice to distribute the mass inside each square $[\tfrac{i-1}{n},\tfrac{i}{n}] \times [\tfrac{j-1}{n},\tfrac{j}{n}]$ (with known value but unknown distribution) uniformly on the square.
It was shown in \cite{LinWanZhaZha25} that this is an extension with the highest Shannon entropy, hence, it carries the least information about the random vector and is in this sense the least biased.
The use of checkerboard copulas for inference in copula models has also been discussed in \cite{GenNesRem14,GeNeRe17}.
Therefore, in this section we extend random discrete copulas $X_k$ and $Y_k$ to copulas defined on the whole $\II^2$ using bilinear functions.

For any $\pi \in \SK$, define $\displaystyle \Chat_\pi$ as a (piecewise) bilinear extension of the discrete copula  $C_\pi$ to the whole unit square $\II^2$, i.e., a checkerboard copula.  Denote with 
$$\PKhat=\{ \Chat_\pi \colon \II^2 \to \II \mid \pi \in \SK \}$$ 
the set of all such extensions. 

We choose $\Chat_\pi\in \PKhat$ uniformly at random. For every $(u,v) \in \II^2$ we denote by
$$\Xhat_k(u,v)=\Chat_\pi(u,v)$$
the value of the chosen element of $\PKhat$ at the point $(u,v)$. This makes $\Xhat_k(u,v)$ a discrete random variable with range in $\II$ and $\Xhat_k$ a random field over $\II^2$. If $(u,v) \in \DK$, then $\Xhat_k(u,v) = X_k(u,v)$, so we already know its distribution. We will now extend the results of Section~\ref{se:permutation} and determine the distribution, expected value and variance of $\Xhat_k(u,v)$ for all $(u,v) \in \II^2$.

\begin{theorem}\label{th:exk_hat}
    Let $k\ge 2$. For any $(u,v) \in \II^2$ let 
$$u = \tfrac{i}{k} + \tfrac{t}{k} \qquad \text{and} \qquad v = \tfrac{j}{k} + \tfrac{s}{k},$$
where $i=\lfloor uk \rfloor \in [k]_0$, $j=\lfloor vk \rfloor \in[k]_0$, $t = uk -i \in [0,1)$, and $s = vk -j \in [0,1)$. Then 
    $$ \EE[\Xhat_k(u,v)] = uv $$ and $$ \VV[\Xhat_k(u,v)] = \frac{1}{(k-1)}\Big(u(1-u)-\frac{t(1-t)}{k}\Big)\Big(v(1-v)-\frac{s(1-s)}{k}\Big).$$
\end{theorem}

\begin{proof}
Since $\Chat_\pi$ is a bilinear extension of $C_\pi$, for every $(u,v) \in \II^2$ the value $\Chat_\pi(u,v)$ is a convex combination of values $C_\pi(\tfrac{i}{k},\tfrac{j}{k})$, $C_\pi(\tfrac{i+1}{k},\tfrac{j}{k})$, $C_\pi(\tfrac{i}{k},\tfrac{j+1}{k})$, and $C_\pi(\tfrac{i+1}{k},\tfrac{j+1}{k})$
$$\Chat_\pi(u,v) = (1-t)(1-s)C_\pi(\tfrac{i}{k},\tfrac{j}{k}) + t(1-s)C_\pi(\tfrac{i+1}{k},\tfrac{j}{k}) + (1-t)sC_\pi(\tfrac{i}{k},\tfrac{j+1}{k}) +tsC_\pi(\tfrac{i+1}{k},\tfrac{j+1}{k}).$$
This implies that
\begin{align*}
    \Xhat_k(u,v) &= (1-t)(1-s)X_k(\tfrac{i}{k},\tfrac{j}{k}) + t(1-s)X_k(\tfrac{i+1}{k},\tfrac{j}{k}) \\
    & \qquad + (1-t)sX_k(\tfrac{i}{k},\tfrac{j+1}{k}) +tsX_k(\tfrac{i+1}{k},\tfrac{j+1}{k}).
\end{align*}
It follows from Theorem~\ref{th:exk} that the expected value of $\Xhat_k(u,v)$ is equal to
$$ \EE\big[\Xhat_k(u,v)\big] = (1-t)(1-s)\tfrac{ij}{k^2} + t(1-s)\tfrac{(i+1)j}{k^2} + (1-t)s\tfrac{i(j+1)}{k^2} +ts\tfrac{(i+1)(j+1)}{k^2}. $$
A straightforward algebraic manipulation reveals that the above expression simplifies to
$$ \EE\big[\Xhat_k(u,v)\big] =  \tfrac{(i+t)(j+s)}{k^2} = uv. $$

In order to compute the variance of $\Xhat_k(u,v)$, we need to find conditional distribution of  $X_k(\tfrac{i+1}{k},\tfrac{j}{k})$, $X_k(\tfrac{i}{k},\tfrac{j+1}{k})$, and $X_k(\tfrac{i+1}{k},\tfrac{j+1}{k})$ under the condition that $X_k(\tfrac{i}{k},\tfrac{j}{k}) = \tfrac{\ell}{k}$, since these random variables are not independent. 

Suppose that $X_k(\tfrac{i}{k},\tfrac{j}{k}) = \tfrac{\ell}{k}$. We may assume without loss of generality that $i<k$ and $j<k$. Recall the notation 
$$L_{ij\ell} = \{\pi \in \SK \mid C_\pi(\tfrac{i}{k},\tfrac{j}{k}) = \tfrac{\ell}{k}\}. $$
We split this set into five subsets
\begin{align*}
   L_{ij\ell}^1 &= \{\pi \in L_{ij\ell} \mid \pi(i+1)\le j, \pi^{-1}(j+1)\le i\}, \\
   L_{ij\ell}^2 &= \{\pi \in L_{ij\ell} \mid \pi(i+1)\le j, \pi^{-1}(j+1)\ge i+2\}, \\
   L_{ij\ell}^3 &= \{\pi \in L_{ij\ell} \mid \pi(i+1)\ge j+2, \pi^{-1}(j+1)\le i\}, \\
   L_{ij\ell}^4 &= \{\pi \in L_{ij\ell} \mid \pi(i+1)\ge j+2, \pi^{-1}(j+1)\ge i+2\}, \\
   L_{ij\ell}^5 &= \{\pi \in L_{ij\ell} \mid \pi(i+1) = j+1\}. 
\end{align*}
It is obvious that these sets are disjoint and their union is $L_{ij\ell}$. If $\pi \in L_{ij\ell}^1$, then 
$$|\{m \in [k] \mid m\le i, \pi(m) \le j\}| = \ell,$$
thus
\begin{align*}
    |\{m \in [k] \mid m\le i+1, \pi(m) \le j\}| &= \ell+1, \\
    |\{m \in [k] \mid m\le i, \pi(m) \le j+1\}| &= \ell+1,  \text{ and}\\
    |\{m \in [k] \mid m\le i+1, \pi(m) \le j+1\}| &= \ell+2. \\
\end{align*} 
It follows that
\begin{equation} \label{eq:Lijl1}
    \pi \in L_{ij\ell}^1 \Leftrightarrow C_\pi(\tfrac{i+1}{k},\tfrac{j}{k}) = \tfrac{\ell+1}{k},  C_\pi(\tfrac{i}{k},\tfrac{j+1}{k}) = \tfrac{\ell+1}{k},  C_\pi(\tfrac{i+1}{k},\tfrac{j+1}{k}) = \tfrac{\ell+2}{k}.
\end{equation} 
Similarly, we obtain
\begin{align*}
   \pi \in L_{ij\ell}^2 &\Leftrightarrow C_\pi(\tfrac{i+1}{k},\tfrac{j}{k}) = \tfrac{\ell+1}{k},  C_\pi(\tfrac{i}{k},\tfrac{j+1}{k}) = \tfrac{\ell}{k},  C_\pi(\tfrac{i+1}{k},\tfrac{j+1}{k}) = \tfrac{\ell+1}{k}, \\
   \pi \in L_{ij\ell}^3 &\Leftrightarrow C_\pi(\tfrac{i+1}{k},\tfrac{j}{k}) = \tfrac{\ell}{k},  C_\pi(\tfrac{i}{k},\tfrac{j+1}{k}) = \tfrac{\ell+1}{k},  C_\pi(\tfrac{i+1}{k},\tfrac{j+1}{k}) = \tfrac{\ell+1}{k}, \\
   \pi \in L_{ij\ell}^4 &\Leftrightarrow C_\pi(\tfrac{i+1}{k},\tfrac{j}{k}) = \tfrac{\ell}{k},  C_\pi(\tfrac{i}{k},\tfrac{j+1}{k}) = \tfrac{\ell}{k},  C_\pi(\tfrac{i+1}{k},\tfrac{j+1}{k}) = \tfrac{\ell}{k}, \\
   \pi \in L_{ij\ell}^5 &\Leftrightarrow C_\pi(\tfrac{i+1}{k},\tfrac{j}{k}) = \tfrac{\ell}{k},  C_\pi(\tfrac{i}{k},\tfrac{j+1}{k}) = \tfrac{\ell}{k},  C_\pi(\tfrac{i+1}{k},\tfrac{j+1}{k}) = \tfrac{\ell+1}{k}.
\end{align*}
We now compute the cardinality of $ L_{ij\ell}^1$.
To obtain a permutation $\pi\in L_{ij\ell}^1$, we first choose a subset $A \subseteq [i]$ containing $\ell$ elements and a subset $B \subseteq [j]$ containing $\ell$ elements and choose any bijective mapping from $A$ to $B$. Next we choose an element $a \in [i] \setminus A$ and map $a$ to $j+1$. Then we choose a subset $C \subseteq [k]\setminus [j+1]$ containing $i-\ell-1$ elements and choose any bijective mapping from $[i]\setminus (A\cup\{a\})$ to $C$. Next we choose an element $b \in [j] \setminus B$ and map $i+1$ to $b$. Then we choose a subset $D \subseteq [k]\setminus [i+1]$ containing $j-\ell-1$ elements and choose any bijective mapping from $D$ to $[j]\setminus (B\cup\{b\})$. Finally, we choose any bijective mapping from $[k]\setminus([i+1]\cup D)$ to $[k]\setminus([j+1]\cup C)$. Counting these choices, it follows that 
\begin{align*}
   |L_{ij\ell}^1| &= \binom{i}{\ell} \, \binom{j}{\ell} \, \ell! \, (i-\ell) \, \binom{k-j-1}{i-\ell-1} \, (i-\ell-1)! \, (j-\ell) \, \binom{k-i-1}{j-\ell-1} \, (j-\ell-1)! \,(k+\ell-i-j)! \\ 
   &= \frac{i! \,j! \, \ell!\, (i-\ell) \, (k-j-1)! \,(i-\ell-1)! \,(j-\ell) \, (k-i-1)! \,(j-\ell-1)! \, (k+\ell-i-j)!}{\ell! \, (i-\ell)! \, \ell! \, (j-\ell) \, (i-\ell-1)! \, (k+\ell-i-j)! \, (j-\ell-1)! \, (k+\ell-i-j)!} \\
   &= \frac{i! \, j! \, (k-i-1)! \, (k-j-1)!}{(i-\ell-1)! \, (j-\ell-1)! \, (k+\ell-i-j)! \, \ell!}. 
\end{align*}
We compute the cardinalities of the remaining sets in a similar way and obtain
\begin{align*}
   |L_{ij\ell}^2| &= \binom{i}{\ell} \, \binom{j}{\ell} \, \ell! \, \binom{k-j-1}{i-\ell} \, (i-\ell)! \, (j-\ell) \, \binom{k-i-1}{j-\ell-1} \, (j-\ell-1)! \,(k+\ell-i-j)! \\ 
   &= \frac{i! \, j! \, (k-i-1)! \, (k-j-1)!}{(i-\ell)! \, (j-\ell-1)! \, (k+\ell-i-j-1)! \, \ell!}, \\ 
   |L_{ij\ell}^3| &= \binom{i}{\ell} \, \binom{j}{\ell} \, \ell! \, (i-\ell) \, \binom{k-j-1}{i-\ell-1} \, (i-\ell-1)! \, \binom{k-i-1}{j-\ell} \, (j-\ell)! \,(k+\ell-i-j)! \\ 
   &= \frac{i! \, j! \, (k-i-1)! \, (k-j-1)!}{(i-\ell-1)! \, (j-\ell)! \, (k+\ell-i-j-1)! \, \ell!}, \\ 
   |L_{ij\ell}^4| &= \binom{i}{\ell} \, \binom{j}{\ell} \, \ell! \, \binom{k-j-1}{i-\ell} \, (i-\ell)! \, \binom{k-i-1}{j-\ell} \, (j-\ell)! \,(k+\ell-i-j-1)^2 \,(k+\ell-i-j-2)! \\ 
   &= \frac{i! \, j! \, (k-i-1)! \, (k-j-1)!}{(i-\ell)! \, (j-\ell)! \, (k+\ell-i-j-2)! \, \ell!}, \\ 
   |L_{ij\ell}^5| &= \binom{i}{\ell} \, \binom{j}{\ell} \, \ell! \, \binom{k-j-1}{i-\ell} \, (i-\ell)! \, \binom{k-i-1}{j-\ell} \, (j-\ell)! \,(k+\ell-i-j-1)! \\ 
   &= \frac{i! \, j! \, (k-i-1)! \, (k-j-1)!}{(i-\ell)! \, (j-\ell)! \, (k+\ell-i-j-1)! \, \ell!}.
\end{align*}
Using characterization \eqref{eq:Lijl1} and equation~\eqref{eq:L_ijl_size}, we compute the conditional probability
\begin{align*}
   p_{ij\ell}^1 &:=\PP\big[X_k(\tfrac{i+1}{k},\tfrac{j}{k}) = \tfrac{\ell+1}{k},  X_k(\tfrac{i}{k},\tfrac{j+1}{k}) = \tfrac{\ell+1}{k},  X_k(\tfrac{i+1}{k},\tfrac{j+1}{k}) = \tfrac{\ell+2}{k} \big| X_k(\tfrac{i}{k},\tfrac{j}{k}) = \tfrac{\ell}{k}\big] \\
   & \phantom{:}= \frac{|L_{ij\ell}^1|}{|L_{ij\ell}|} = \frac{\frac{i! \, j! \, (k-i-1)! \, (k-j-1)!}{(i-\ell-1)! \, (j-\ell-1)! \, (k+\ell-i-j)! \, \ell!}}{\frac{i! \, j! \, (k-i)! \, (k-j)!}{(i-\ell)! \, (j-\ell)! \, (k+\ell-i-j)! \, \ell!}} =\  \frac{(i-\ell) \, (j-\ell)}{(k-i) \, (k-j)}.
\end{align*}
Similarly,
\begin{align*}
   p_{ij\ell}^2 &:=\PP\big[X_k(\tfrac{i+1}{k},\tfrac{j}{k}) = \tfrac{\ell+1}{k},  X_k(\tfrac{i}{k},\tfrac{j+1}{k}) = \tfrac{\ell}{k}, X_k(\tfrac{i+1}{k},\tfrac{j+1}{k}) = \tfrac{\ell+1}{k} \big| X_k(\tfrac{i}{k},\tfrac{j}{k}) = \tfrac{\ell}{k}\big] \\
   & \phantom{:}= \frac{|L_{ij\ell}^2|}{|L_{ij\ell}|} = \frac{(j-\ell) \, (k+\ell-i-j)}{(k-i) \, (k-j)} , \\
   p_{ij\ell}^3 &:=\PP\big[X_k(\tfrac{i+1}{k},\tfrac{j}{k}) = \tfrac{\ell}{k},  X_k(\tfrac{i}{k},\tfrac{j+1}{k}) = \tfrac{\ell+1}{k},  X_k(\tfrac{i+1}{k},\tfrac{j+1}{k}) = \tfrac{\ell+1}{k} \big| X_k(\tfrac{i}{k},\tfrac{j}{k}) = \tfrac{\ell}{k}\big] \\
   & \phantom{:}= \frac{|L_{ij\ell}^3|}{|L_{ij\ell}|} = \frac{(i-\ell) \, (k+\ell-i-j)}{(k-i) \, (k-j)} , \\
   p_{ij\ell}^4 &:=\PP\big[X_k(\tfrac{i+1}{k},\tfrac{j}{k}) = \tfrac{\ell}{k},  X_k(\tfrac{i}{k},\tfrac{j+1}{k}) = \tfrac{\ell}{k},  X_k(\tfrac{i+1}{k},\tfrac{j+1}{k}) = \tfrac{\ell}{k} \big| X_k(\tfrac{i}{k},\tfrac{j}{k}) = \tfrac{\ell}{k}\big] \\
   & \phantom{:}= \frac{|L_{ij\ell}^4|}{|L_{ij\ell}|} = \frac{(k+\ell-i-j) \, (k+\ell-i-j-1)}{(k-i) \, (k-j)} , \\
   p_{ij\ell}^5 &:=\PP\big[X_k(\tfrac{i+1}{k},\tfrac{j}{k}) = \tfrac{\ell}{k},  X_k(\tfrac{i}{k},\tfrac{j+1}{k}) = \tfrac{\ell}{k},  X_k(\tfrac{i+1}{k},\tfrac{j+1}{k}) = \tfrac{\ell+1}{k} \big| X_k(\tfrac{i}{k},\tfrac{j}{k}) = \tfrac{\ell}{k}\big] \\
   & \phantom{:}= \frac{|L_{ij\ell}^5|}{|L_{ij\ell}|} = \frac{(k+\ell-i-j)}{(k-i) \, (k-j)} .
\end{align*}
If $\pi \in L_{ij\ell}^1$, condition \eqref{eq:Lijl1} implies that $\Chat_\pi(u,v) =  \frac{\ell+t+s}{k}$. If $\pi \in L_{ij\ell}^2$, we have $\Chat_\pi(u,v) =  \frac{\ell+t}{k}$, if $\pi \in L_{ij\ell}^3$, we obtain $\Chat_\pi(u,v) =  \frac{\ell+s}{k}$, if $\pi \in L_{ij\ell}^4$, we have $\Chat_\pi(u,v) =  \frac{\ell}{k}$, and, finally, if $\pi \in L_{ij\ell}^5$, we get $\Chat_\pi(u,v) =  \frac{\ell+ts}{k}$.
So, the conditional distribution of random variable $\Xhat_k(u,v)$ under the condition that $X_k(\tfrac{i}{k},\tfrac{j}{k}) = \tfrac{\ell}{k}$ is
$$\big(\Xhat_k(u,v) \big| X_k(\tfrac{i}{k},\tfrac{j}{k}) = \tfrac{\ell}{k}\big) \sim  \begin{pmatrix}
    \frac{\ell+t+s}{k} & \frac{\ell+t}{k} & \frac{\ell+s}{k} & \frac{\ell}{k} & \frac{\ell+ts}{k} \\[2mm]
    p_{ij\ell}^1 & p_{ij\ell}^2 & p_{ij\ell}^3 & p_{ij\ell}^4 & p_{ij\ell}^5
\end{pmatrix}.$$
We remark that in this formula some of the values may collapse into a single one. For example, when $s=t$, we have $\frac{\ell+t}{k}=\frac{\ell+s}{k}$, and this value is attained with probability $p_{ij\ell}^2+p_{ij\ell}^3$.
We denote the sets
$$M_{st\ell} = \{\tfrac{\ell+t+s}{k}, \tfrac{\ell+t}{k}, \tfrac{\ell+s}{k}, \tfrac{\ell}{k}, \tfrac{\ell+ts}{k}\} \qquad \text{and} \qquad M_{uv} = \bigcup_{\ell=\max \{0,  i+j-k \}}^{\min \{i,j \}} M_{st\ell}.$$
We have proved that $\Xhat_k(u,v)$ is a discrete random variable with range in $M_{uv}$. Note that in the case $t > 0$, $s > 0$, $t\ne s$, and $t+s \ne 1$ each set $M_{st\ell}$ contains five different elements and the sets $M_{st\ell}$ are pairwise disjoint, so that 
\begin{align*}
    \PP\big[\Xhat_k(u,v) = \tfrac{\ell+t+s}{k}\big] &= p_{ij\ell}^1\, \PP[X_k(\tfrac{i}{k},\tfrac{j}{k}) = \tfrac{\ell}{k}], \\
    \PP\big[\Xhat_k(u,v) = \tfrac{\ell+t}{k}\big] &= p_{ij\ell}^2\, \PP[X_k(\tfrac{i}{k},\tfrac{j}{k}) = \tfrac{\ell}{k}], \\
    \PP\big[\Xhat_k(u,v) = \tfrac{\ell+s}{k}\big] &= p_{ij\ell}^3\, \PP[X_k(\tfrac{i}{k},\tfrac{j}{k}) = \tfrac{\ell}{k}], \\
    \PP\big[\Xhat_k(u,v) = \tfrac{\ell}{k}\big] &= p_{ij\ell}^4\, \PP[X_k(\tfrac{i}{k},\tfrac{j}{k}) = \tfrac{\ell}{k}], \text{ and}\\
    \PP\big[\Xhat_k(u,v) = \tfrac{\ell+ts}{k}\big] &= p_{ij\ell}^5\, \PP[X_k(\tfrac{i}{k},\tfrac{j}{k}) = \tfrac{\ell}{k}]. 
\end{align*}
If $0 < t = s \ne \tfrac12$ each set $M_{st\ell}$ contains four different elements and the sets $M_{st\ell}$ are still pairwise disjoint, so that in this case
$$\PP\big[\Xhat_k(u,v) = \tfrac{\ell+t}{k}\big] = (p_{ij\ell}^2 + p_{ij\ell}^3)\, \PP[X_k(\tfrac{i}{k},\tfrac{j}{k}) = \tfrac{\ell}{k}],$$
and the other probabilities are the same as above. If $t > 0, s > 0$, $t\ne s$, and $t+s=1$ the sets are not disjoint, namely $M_{st\ell} \cap M_{st(\ell+1)} = \{\tfrac{\ell+1}{k}\}$, so that in this case
$$\PP\big[\Xhat_k(u,v) = \tfrac{\ell+1}{k}\big] = p_{ij\ell}^1\, \PP[X_k(\tfrac{i}{k},\tfrac{j}{k}) = \tfrac{\ell}{k}] + p_{ij(\ell+1)}^4\, \PP[X_k(\tfrac{i}{k},\tfrac{j}{k}) = \tfrac{\ell+1}{k}],$$
and the other probabilities are the same as above. If $t=0$ or $s=0$ or $t=s=\tfrac12$ even more possible values of $\Xhat_k(u,v)$ collapse into a single value. 

Furthermore, conditional probabilities of the values of $X_k(\tfrac{i+1}{k},\tfrac{j}{k})$, $X_k(\tfrac{i}{k},\tfrac{j+1}{k})$, and $X_k(\tfrac{i+1}{k},\tfrac{j+1}{k})$ under the condition that $X_k(\tfrac{i}{k},\tfrac{j}{k})  = \tfrac{\ell}{k}$ are
\begin{align*}
   \PP\big[X_k(\tfrac{i+1}{k},\tfrac{j}{k}) = \tfrac{\ell}{k} \big| X_k(\tfrac{i}{k},\tfrac{j}{k}) = \tfrac{\ell}{k}\big] 
   &= p_{ij\ell}^3 + p_{ij\ell}^4 + p_{ij\ell}^5 = \frac{k+\ell-i-j}{k-i}, \\
   \PP\big[X_k(\tfrac{i+1}{k},\tfrac{j}{k}) = \tfrac{\ell+1}{k} \big| X_k(\tfrac{i}{k},\tfrac{j}{k}) = \tfrac{\ell}{k}\big] 
   &= p_{ij\ell}^1 + p_{ij\ell}^2 = \frac{j-\ell}{k-i}, \\
   \PP\big[X_k(\tfrac{i}{k},\tfrac{j+1}{k}) = \tfrac{\ell}{k} \big| X_k(\tfrac{i}{k},\tfrac{j}{k}) = \tfrac{\ell}{k}\big] 
   &= p_{ij\ell}^2 + p_{ij\ell}^4 + p_{ij\ell}^5 = \frac{k+\ell-i-j}{k-j}, \\
   \PP\big[X_k(\tfrac{i}{k},\tfrac{j+1}{k}) = \tfrac{\ell+1}{k} \big| X_k(\tfrac{i}{k},\tfrac{j}{k}) = \tfrac{\ell}{k}\big] 
   &= p_{ij\ell}^1 + p_{ij\ell}^3 = \frac{i-\ell}{k-j}, \\
   \PP\big[X_k(\tfrac{i+1}{k},\tfrac{j+1}{k}) = \tfrac{\ell}{k} \big| X_k(\tfrac{i}{k},\tfrac{j}{k}) = \tfrac{\ell}{k}\big] 
   &= p_{ij\ell}^4 = \frac{(k+\ell-i-j) \, (k+\ell-i-j-1)}{(k-i) \, (k-j)}, \\
   \PP\big[X_k(\tfrac{i+1}{k},\tfrac{j+1}{k}) = \tfrac{\ell+1}{k} \big| X_k(\tfrac{i}{k},\tfrac{j}{k}) = \tfrac{\ell}{k}\big] 
   &= p_{ij\ell}^2 + p_{ij\ell}^3 + p_{ij\ell}^5 = \frac{(k+\ell-i-j) \, (i+j+1-2\ell)}{(k-i) \, (k-j)}, \\
   \PP\big[X_k(\tfrac{i+1}{k},\tfrac{j+1}{k}) = \tfrac{\ell+2}{k} \big| X_k(\tfrac{i}{k},\tfrac{j}{k}) = \tfrac{\ell}{k}\big] 
   &= p_{ij\ell}^1 = \frac{(i-\ell) \, (j-\ell)}{(k-i) \, (k-j)}.
\end{align*}

We are now able to compute the covariances between random variables $X_k(\tfrac{i}{k},\tfrac{j}{k})$, $X_k(\tfrac{i+1}{k},\tfrac{j}{k})$, $X_k(\tfrac{i}{k},\tfrac{j+1}{k})$, and $X_k(\tfrac{i+1}{k},\tfrac{j+1}{k})$.
We have
\begin{align*}
    \EE\big[X_k(\tfrac{i}{k},\tfrac{j}{k})X_k(\tfrac{i+1}{k},\tfrac{j}{k})\big] &= \sum_{\ell=\max \{0,  i+j-k \}}^{\min \{i,j \}} \left(\frac{\ell^2}{k^2}\cdot \PP\big[X_k(\tfrac{i}{k},\tfrac{j}{k}) = \tfrac{\ell}{k}, X_k(\tfrac{i+1}{k},\tfrac{j}{k}) = \tfrac{\ell}{k}\big] \right.\\
    & \left.\qquad + \frac{\ell(\ell+1)}{k^2}\cdot \PP\big[X_k(\tfrac{i}{k},\tfrac{j}{k}) = \tfrac{\ell}{k}, X_k(\tfrac{i+1}{k},\tfrac{j}{k})= \tfrac{\ell+1}{k}\big]\right)\\
    &= \sum_{\ell=\max \{0,  i+j-k \}}^{\min \{i,j \}} \frac{\ell^2}{k^2}\cdot \PP\big[X_k(\tfrac{i}{k},\tfrac{j}{k}) = \tfrac{\ell}{k}\big] \\
    & \qquad + \sum_{\ell=\max \{0,  i+j-k \}}^{\min \{i,j \}} \frac{\ell}{k^2}\cdot \PP\big[X_k(\tfrac{i}{k},\tfrac{j}{k}) = \tfrac{\ell}{k}, X_k(\tfrac{i+1}{k},\tfrac{j}{k})= \tfrac{\ell+1}{k}\big].
\end{align*}
The first sum is equal to $\EE\big[X_k(\tfrac{i}{k},\tfrac{j}{k})^2\big]$, while in the second sum we use the fact that the probability is equal to $(p_{ij\ell}^1 + p_{ij\ell}^2)\PP\big[X_k(\tfrac{i}{k},\tfrac{j}{k}) = \tfrac{\ell}{k}\big]$.
Thus, we obtain
\begin{align*}
    \EE\big[X_k(\tfrac{i}{k},\tfrac{j}{k})X_k(\tfrac{i+1}{k},\tfrac{j}{k})\big] &= \EE\big[X_k(\tfrac{i}{k},\tfrac{j}{k})^2\big] + \sum_{\ell=\max \{0,  i+j-k \}}^{\min \{i,j \}} \frac{\ell}{k^2}\cdot \frac{j-\ell}{k-i}\cdot \PP\big[X_k(\tfrac{i}{k},\tfrac{j}{k}) = \tfrac{\ell}{k}\big]\\
    &= \EE\big[X_k(\tfrac{i}{k},\tfrac{j}{k})^2\big] + \frac{j}{k(k-i)}\sum_{\ell=\max \{0,  i+j-k \}}^{\min \{i,j \}} \frac{\ell}{k}\cdot \PP\big[X_k(\tfrac{i}{k},\tfrac{j}{k}) = \tfrac{\ell}{k}\big] \\
    & \qquad - \frac{1}{k-i}\sum_{\ell=\max \{0,  i+j-k \}}^{\min \{i,j \}} \frac{\ell^2}{k^2}\cdot \PP\big[X_k(\tfrac{i}{k},\tfrac{j}{k}) = \tfrac{\ell}{k}\big].
\end{align*}
Again, the first sum is equal to $\EE\big[X_k(\tfrac{i}{k},\tfrac{j}{k})\big]$, and the second to $\EE\big[X_k(\tfrac{i}{k},\tfrac{j}{k})^2\big]$, so by~\eqref{eq:E(X_k^2)} we have
\begin{align*}
    \EE\big[X_k(\tfrac{i}{k},\tfrac{j}{k})X_k(\tfrac{i+1}{k},\tfrac{j}{k})\big] &= \EE\big[X_k(\tfrac{i}{k},\tfrac{j}{k})^2\big] + \frac{j}{k(k-i)}\EE\big[X_k(\tfrac{i}{k},\tfrac{j}{k})\big]
    - \frac{1}{k-i}\EE\big[X_k(\tfrac{i}{k},\tfrac{j}{k})^2\big] \\
    &=  \frac{k-i-1}{k-i} \cdot \left(\frac{ij}{k^3} + \frac{i(i-1)j(j-1)}{k^3 (k-1)}\right) + \frac{j}{k(k-i)}\cdot \frac{ij}{k^2} \\
    &=  \frac{ij(k-i-1+ij)}{k^3(k-1)}.
\end{align*}
This implies that
\begin{align*}
    \Cov[X_k(\tfrac{i}{k},\tfrac{j}{k}), X_k(\tfrac{i+1}{k},\tfrac{j}{k})] &= \EE\big[X_k(\tfrac{i}{k},\tfrac{j}{k})X_k(\tfrac{i+1}{k},\tfrac{j}{k})\big] - \EE[X_k(\tfrac{i}{k},\tfrac{j}{k})]\EE[X_k(\tfrac{i+1}{k},\tfrac{j}{k})] \\
    &=  \frac{ij(k-i-1+ij)}{k^3(k-1)} - \frac{ij}{k^2}\cdot\frac{(i+1)j}{k^2} \\
    &=  \frac{ij(k-i-1)(k-j)}{k^4(k-1)}. 
\end{align*}
Similarly, we obtain
$$\Cov[X_k(\tfrac{i}{k},\tfrac{j}{k}), X_k(\tfrac{i}{k},\tfrac{j+1}{k})] = \frac{ij(k-i)(k-j-1)}{k^4(k-1)}, $$
$$\Cov[X_k(\tfrac{i}{k},\tfrac{j}{k}), X_k(\tfrac{i+1}{k},\tfrac{j+1}{k})] = \frac{ij(k-i-1)(k-j-1)}{k^4(k-1)}, $$
and
$$\Cov[X_k(\tfrac{i+1}{k},\tfrac{j}{k}), X_k(\tfrac{i}{k},\tfrac{j+1}{k})] = \frac{ij(k-i-1)(k-j-1)}{k^4(k-1)}. $$
Finally, expanding the variance of $\Xhat_k(u,v)$, we obtain
\begin{align*}
    \VV\big[\Xhat_k(u,v)\big] &= (1-t)^2(1-s)^2\cdot\VV[X_k(\tfrac{i}{k},\tfrac{j}{k})] + t^2(1-s)^2\cdot\VV[X_k(\tfrac{i+1}{k},\tfrac{j}{k})] \\
    & \qquad + (1-t)^2s^2\cdot\VV[X_k(\tfrac{i}{k},\tfrac{j+1}{k})]+t^2s^2 \cdot\VV[X_k(\tfrac{i+1}{k},\tfrac{j+1}{k})] \\
    & \qquad + 2(1-t)t(1-s)^2\cdot\Cov[X_k(\tfrac{i}{k},\tfrac{j}{k}), X_k(\tfrac{i+1}{k},\tfrac{j}{k})]  \\
    & \qquad + 2(1-t)ts^2\cdot\Cov[X_k(\tfrac{i}{k},\tfrac{j+1}{k}), X_k(\tfrac{i+1}{k},\tfrac{j+1}{k})] \\
    & \qquad + 2(1-t)^2(1-s)s\cdot\Cov[X_k(\tfrac{i}{k},\tfrac{j}{k}), X_k(\tfrac{i}{k},\tfrac{j+1}{k})]  \\
    & \qquad + 2t^2(1-s)s\cdot\Cov[X_k(\tfrac{i+1}{k},\tfrac{j}{k}), X_k(\tfrac{i+1}{k},\tfrac{j+1}{k})] \\
    & \qquad + 2(1-t)t(1-s)s\cdot\Cov[X_k(\tfrac{i}{k},\tfrac{j}{k}), X_k(\tfrac{i+1}{k},\tfrac{j+1}{k})]  \\
    & \qquad + 2(1-t)t(1-s)s\cdot\Cov[X_k(\tfrac{i+1}{k},\tfrac{j}{k}), X_k(\tfrac{i}{k},\tfrac{j+1}{k})].
\end{align*}
Plugging in the expressions for the (co-) variances and simplifying, we get
\begin{align*}
    \VV\big[\Xhat_k(u,v)\big]
    &=\frac{\big((i+t)(k-i-t)-kt(1-t)\big)\big((j+s)(k-j-s)-ks(1-s)\big)}{k^4(k-1)} \\
    &= \frac{1}{(k-1)}\Big(u(1-u)-\frac{t(1-t)}{k}\Big)\Big(v(1-v)-\frac{s(1-s)}{k}\Big),
\end{align*}
which finishes the proof.
\end{proof}

In a similar manner as for random discrete copula $X_k$, we extend random discrete copula $Y_k$ to its bilinear extension $\Yhat_k$, which is defined for all $(u,v) \in \II^2$.
Recall that for $(u,v) \in \DK$ we have $Y_k(u,v)= \sum_{\pi \in \SK} \alpha_\pi \, C_{\pi}(u,v)$, where $k!$-tuple $\aaa=(\alpha_\pi)_{\pi \in \SK}$ is distributed uniformly over the simplex $\sigma^{k!-1}$.
Let $$\Yhat_k(u,v)= \sum_{\pi \in \SK} \alpha_{\pi} \, \Chat_{\pi}(u,v)$$
for all $(u,v) \in \II^2$.
Recall also the notation $u = \tfrac{i}{k} + \tfrac{t}{k}$ and $v = \tfrac{j}{k} + \tfrac{s}{k}$,
where $i=\lfloor uk \rfloor \in [k]_0$, $j=\lfloor vk \rfloor \in [k]_0$, $t = uk -i \in [0,1)$, and $s = vk -j \in [0,1)$. It follows that
$$\Yhat_k(u,v) = (1-t)(1-s)Y_k(\tfrac{i}{k},\tfrac{j}{k}) + t(1-s)Y_k(\tfrac{i+1}{k},\tfrac{j}{k}) + (1-t)sY_k(\tfrac{i}{k},\tfrac{j+1}{k}) +tsY_k(\tfrac{i+1}{k},\tfrac{j+1}{k}).$$
In our last theorem we compute the expected value and variance of random variable $\Yhat_k(u,v)$.

\begin{theorem}\label{th:eyk_hat}
    Let $k\ge 2$. For any $(u,v) \in \II^2$ we have 
    \begin{align*}
    \EE[\Yhat_k(u,v)] &=\EE[\Xhat_k(u,v)]= uv \qquad\text{and}\\
    \VV[\Yhat_k(u,v)] &=\frac{1}{k!+1}\VV[\Xhat_k(u,v)] = \frac{1}{(k!+1)(k-1)}\Big(u(1-u)-\frac{t(1-t)}{k}\Big)\Big(v(1-v)-\frac{s(1-s)}{k}\Big).
    \end{align*}
\end{theorem}

\begin{proof}
By equation~\eqref{eq:EVC_uni} and Theorem~\ref{th:exk_hat} we have
$$\EE\big[\Yhat_k(u,v)\big]= \sum_{\pi \in \SK} \EE[\alpha_{\pi}] \Chat_{\pi}(u,v) = \frac{1}{k!}\sum_{\pi \in \SK} \Chat_{\pi}(u,v) = \EE[\Xhat_k(u,v)]= uv.$$
In order to compute the variance $\VV\big[\Yhat_k(u,v)\big]$, we need the covariances between random variables $Y_k(\tfrac{i}{k},\tfrac{j}{k})$, $Y_k(\tfrac{i+1}{k},\tfrac{j}{k})$, $Y_k(\tfrac{i}{k},\tfrac{j+1}{k})$, and $Y_k(\tfrac{i+1}{k},\tfrac{j+1}{k})$.
By Lemma~\ref{lem:cov_YY} we have
\begin{align*}
    \Cov[Y_k(\tfrac{i}{k},\tfrac{j}{k}), Y_k(\tfrac{i+1}{k},\tfrac{j}{k})] &= \frac{1}{k!+1}\, \Cov[X_k(\tfrac{i}{k},\tfrac{j}{k}), X_k(\tfrac{i+1}{k},\tfrac{j}{k})], \\
    \Cov[Y_k(\tfrac{i}{k},\tfrac{j}{k}), Y_k(\tfrac{i}{k},\tfrac{j+1}{k})] &= \frac{1}{k!+1}\, \Cov[X_k(\tfrac{i}{k},\tfrac{j}{k}), X_k(\tfrac{i}{k},\tfrac{j+1}{k})], \\
    \Cov[Y_k(\tfrac{i}{k},\tfrac{j}{k}), Y_k(\tfrac{i+1}{k},\tfrac{j+1}{k})] &= \frac{1}{k!+1}\, \Cov[X_k(\tfrac{i}{k},\tfrac{j}{k}), X_k(\tfrac{i+1}{k},\tfrac{j+1}{k})], \\
    \Cov[Y_k(\tfrac{i+1}{k},\tfrac{j}{k}), Y_k(\tfrac{i}{k},\tfrac{j+1}{k})] &= \frac{1}{k!+1}\, \Cov[X_k(\tfrac{i+1}{k},\tfrac{j}{k}), X_k(\tfrac{i}{k},\tfrac{j+1}{k})]. 
\end{align*}
Since $\Yhat_k(u,v)$ is obtained from $Y_k(\tfrac{i}{k},\tfrac{j}{k})$, $Y_k(\tfrac{i+1}{k},\tfrac{j}{k})$, $Y_k(\tfrac{i}{k},\tfrac{j+1}{k})$, and $Y_k(\tfrac{i+1}{k},\tfrac{j+1}{k})$ in the same manner as $\Xhat_k(u,v)$ is obtained from $X_k(\tfrac{i}{k},\tfrac{j}{k})$, $X_k(\tfrac{i+1}{k},\tfrac{j}{k})$, $X_k(\tfrac{i}{k},\tfrac{j+1}{k})$, and $X_k(\tfrac{i+1}{k},\tfrac{j+1}{k})$, we finally obtain
\begin{align*}
    \VV\big[\Yhat_k(u,v)\big] &= \frac{1}{k!+1}\, \VV\big[\Xhat_k(u,v)\big] \\
    &= \frac{1}{(k!+1)(k-1)}\Big(u(1-u)-\frac{t(1-t)}{k}\Big)\Big(v(1-v)-\frac{s(1-s)}{k}\Big),
\end{align*}
which finishes the proof.
\end{proof}

As anticipated, the expected values of random variables $\Xhat_k(u,v)$ and $\Yhat_k(u,v)$ are equal to the value of the product copula, i.e., they are given by the same formal expression as expected values of $X_k(u,v)$ and $Y_k(u,v)$. On the other hand, the formal expressions for the variances of  $\Xhat_k(u,v)$ and $\Yhat_k(u,v)$ are more involved than the respective ones for $X_k(u,v)$ and $Y_k(u,v)$. However, their order of magnitude as $k$ grows to infinity remains the same.

\section{An application to concordance measures} \label{se:simulation}

In this section we present a particular example of an application of random discrete copula induced by a permutation to the study of concordance measures.

Concordance measures are used in statistical inference to quantify the degree of association between random variables.
Given a pair of continuous random variables $X$ and $Y$, the value of a concordance measure is uniquely determined by the copula of the pair, \cite{Kr58,Li14}.
Different measures are used in practice, such as Spearman's rho, Kendall's tau, Gini's gamma and Blomqvist's beta, so the relations between them are of interest. In particular, it is important to assess how much these measures can differ in order to decide which one to use.
Several papers have considered the problem of determining the exact region $\{(\kappa_1(C),\kappa_2(C)) \mid C \text{ is a copula}\}$, where $\kappa_1$ and $\kappa_2$ are two concordance measures, see for example \cite{ScPaTr17,KoBuSt23}.

Here, we investigate the pair Spearman's rho, denoted by $\rho$, and Kendall's tau, denoted by $\tau$. The exact region $\Omega_{\tau,\rho}$ for this pair was determined in \cite{ScPaTr17} and is contained in the set
$$\Omega_0=\{(\tau,\rho) \in [-1,1]^2 \mid -1 \le 3\tau-2\rho \le 1,\ \tfrac{1}{2}(1+\tau)^2-1 \le \rho \le 1-\tfrac{1}{2}(1-\tau)^2\},$$
depicted on Figure~\ref{fig:rand_8} (right).
The difference between $\Omega_0$ and the exact region $\Omega_{\tau,\rho}$ is very small, the area of $\Omega_0$ is $\tfrac{7}{6} \approx 1.1667$ while the area of $\Omega_{\tau,\rho}$ is approximately $1.1543$, see \cite{KoBuMo22}.

Note that each point in the region $\Omega_{\tau,\rho}$ may be attained by many distinct copulas.
Loosely speaking, the copulas attaining the points near the boundary of the region seem to be very specific, see \cite{ScPaTr17}, while copulas attaining the points close to the center $(\tau,\rho)=(0,0)$, are very common.
To support this observation, we investigate the probability distribution of the point $(\tau(C),\rho(C))$ where $C$ is an extension to $[0,1]^2$ of a random discrete copula $X_k$ induced by a permutation and $k$ is fixed.
In essence, this corresponds to considering extensions of empirical copulas obtained from samples of size $k$.

The paper \cite{GeNe07} discusses the issue of using copulas to model discrete random vectors. The core problem is that the copula corresponding to a discrete random vector $(X,Y)$ is not uniquely determined, i.e., every discrete copula can be extended to $[0,1]^2$ in many ways.
This in turn means that the pair $(\tau(C),\rho(C))$ depends on the particular extension $C$ of the discrete copula and is not a single point.

Hence, we consider the lower and upper bound for all possible extensions. Let $C_\pi$ be a realization of a random discrete copula $X_k$ induced by a permutation. The upper bound of all possible extensions of $C_\pi$ is obtained by spreading the mass in each square $[\tfrac{i-1}{k},\tfrac{i}{k}] \times [\tfrac{j-1}{k},\tfrac{j}{k}]$ uniformly along the main diagonal of the square (as in $M$), while the lower bound is obtained by spreading the mass in each square uniformly along the opposite diagonal of the square (as in $W$).
We denote these extensions by $C_\pi^M$ and $C_\pi^W$. Note that both of these extensions are shuffles of min, see \cite[\S 3.2.3]{Nel06}. An elementary calculation reveals that
\begin{equation} \label{eq:rho-tau}
\begin{aligned}
    \rho(C_\pi^M) &=\rho(C_\pi^W)+\tfrac{2}{k^2}, & \rho(\widehat{C}_\pi) &=\tfrac{1}{2}(\rho(C_\pi^W)+\rho(C_\pi^M)),\\
    \tau(C_\pi^M) &=\tau(C_\pi^W)+\tfrac{2}{k} , & \tau(\widehat{C}_\pi) &=\tfrac{1}{2}(\tau(C_\pi^W)+\tau(C_\pi^M)),
\end{aligned}
\end{equation}
where $\widehat{C}_\pi$ is the checkerboard extension, see Section~\ref{se:extension}.

\begin{figure}
    \centering
    \includegraphics[width=0.5\linewidth]{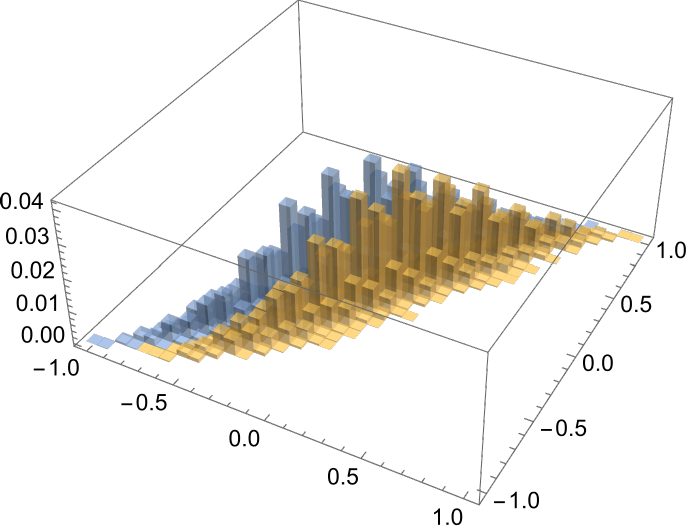}
    \includegraphics[width=0.4\linewidth]{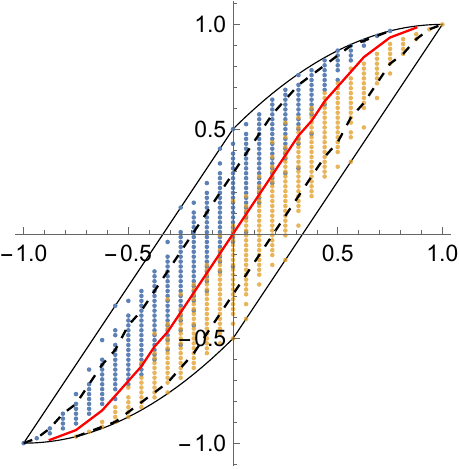}
    \caption{The histograms and the set of points $(\tau(C),\rho(C))$ for extensions $C=C_\pi^W$ and $C=C_\pi^M$ for all $\pi \in \mathcal{S}_8$.}
    \label{fig:rand_8}
\end{figure}

We first consider the mesh $\Delta_k$ for $k=8$. There are $40\,320$ permutations in $\mathcal{S}_8$, so we are able to generate $C_\pi$ for all permutations, and calculate the points $(\tau(C_\pi^W),\rho(C_\pi^W))$ and $(\tau(C_\pi^M),\rho(C_\pi^M))$ for all of them.
The so obtained points are depicted in Figure~\ref{fig:rand_8}, right, against the region $\Omega_0$. The points $(\tau(C_\pi^W),\rho(C_\pi^W))$ are plotted in blue, while the points $(\tau(C_\pi^M),\rho(C_\pi^M))$ are plotted in yellow.
The histograms of their distributions with $32 \times 32$ bins are depicted in Figure~\ref{fig:rand_8}, left, with the same color-coding.
The red curve represents the mean of $\rho(\widehat{C}_\pi)$ given the value of $\tau(\widehat{C}_\pi)$.
Note that by equations~\eqref{eq:rho-tau} the histogram of the distribution of points $(\tau(\widehat{C}_\pi),\rho(\widehat{C}_\pi))$ lies "half-way between" the blue and yellow histograms. The lower dashed curve represent the $5\%$ quantile for the value of $\rho(C_\pi^M)$ given the value of $\tau(C_\pi^M)$, while the upper dashed curve represent the $95\%$ quantile for the value of $\rho(C_\pi^W)$ given the value of $\tau(C_\pi^W)$.
We caution that the interval between the dashed curves cannot be interpreted as the $90\%$ confidence interval of $\rho(C)$ given $\tau(C)$, where $C$ is an extension of $C_\pi$, because the values $\tau(C_\pi^W)$ and $\tau(C_\pi^M)$ differ.

As expected, the simulations support the intuition that the copulas attaining the points close to the boundary of the set $\Omega_0$ are quite rare while those attaining points close to the red curve, i.e., the  mean of $\rho(\widehat{C}_\pi)$ given $\tau(\widehat{C}_\pi)$, are much more common.
The simulation also suggests that, although the range of $\rho$ given the value of $\tau$ may be quite large, we can expect the point $(\tau(C),\rho(C))$ to be close to the red curve when dealing with a random copula $C$.
This means, that a point far away from the red curve indicates a very specific dependence structure. For example, Figure~\ref{fig:boundary_mean} depicts the mean mass distribution of discrete copulas corresponding to the yellow points below the lower dashed line on Figure~\ref{fig:rand_8} for different values of $\tau$.

\begin{figure}
    \centering
    \includegraphics[width=\linewidth]{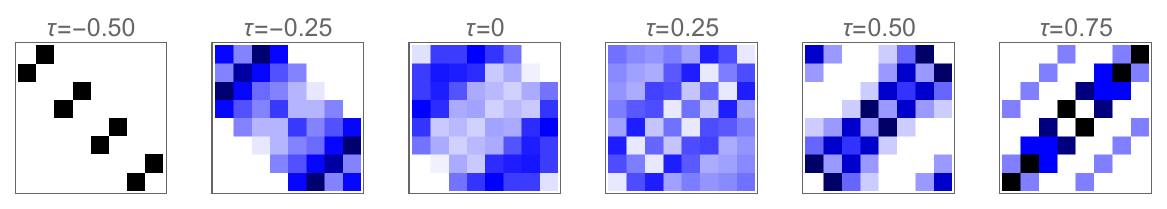}
    \caption{Mean mass distribution of copulas corresponding to yellow points below the lower dashed line in Figure~\ref{fig:rand_8} for different values of $\tau$. Darker color indicates more mass.}
    \label{fig:boundary_mean}
\end{figure}

In the second example, we analyze the mesh $\Delta_k$ for $k=12$. There are too many permutations in $\mathcal{S}_{12}$ to consider all of them, so we generated a random sample of size $380 \, 000$ which is approximately $0.079\%$ of all permutations. We repeated all the calculations of the previous example. The results are depicted in Figure~\ref{fig:rand_12}. In this case, the histograms were chosen to have $72 \times 72$
bins.

\begin{figure}
    \centering
    \includegraphics[width=0.5\linewidth]{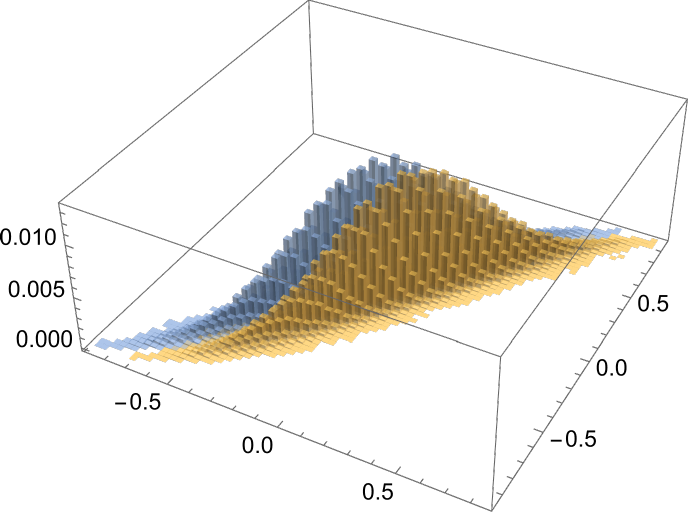}
    \includegraphics[width=0.4\linewidth]{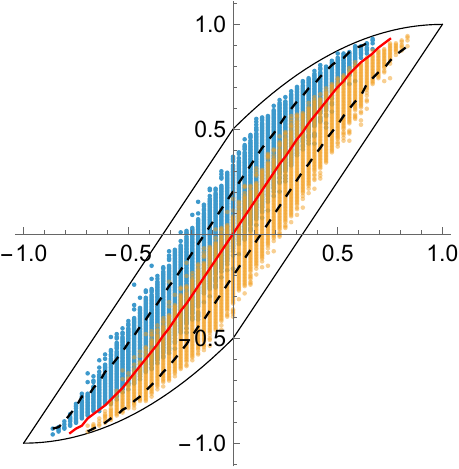}
    \caption{The histograms and the set of points $(\tau(C),\rho(C))$ for extensions $C=C_\pi^W$ and $C=C_\pi^M$ for a sample of $380\,000$ permutations $\pi \in \mathcal{S}_{12}$.}
    \label{fig:rand_12}
\end{figure}

\section{Concluding remarks} \label{se:concluding}

In this paper we
introduced the notion of a random discrete copula and explored its basic properties. 
We proved that the expected value of the evaluation of a random discrete copula at a certain point is equal to the value of the product copula, while its variance rapidly decreased with the size of the mesh.  This seems to support the intuition - if we imagine picking two real life random variables at random with no preferences, intuitively it is reasonable to expect that they would be nearly independent with high probability, since they could be describing two unrelated phenomena. 
The expected value is also consistent with our choice to model the situation where no information is available (broadness).
The fast rate of convergence of the variance to $0$ with the mesh size is a consequence of our decision to define a random discrete copula on $\DK$ as a random convex combination of extreme points of the Birkhoff polytope $\mathscr{B}_k$. The coefficients were chosen uniformly with respect to the Lebesgue measure on a standard simplex. 
This means that the copulas corresponding to the points near the center of mass of the Birkhoff polytope $\mathscr{B}_k$ are much more probable than those close to the boundary.
An important point to take away from our results is that the shape of variance as a function of the point $(u,v)$ does not change with the mesh size in the discrete case. It contains the crucial factor $uv(1-u)(1-v)$ which is only scaled with respect to the mesh size.
We also considered bilinear extension of a random discrete copula and calculated its expected value and variance at every point in the unit square.

Our future work will include investigation of the convergence of a random discrete copula as the size of the mesh increases. It would also be interesting to consider other distributions on the set of discrete copulas and look for one with a slower rate of convergence of the variance.
Furthermore, we plan to develop models that will account for prior knowledge about the unknown dependence structure. For example, one might want the model to reflect a prior knowledge about concordance, tail dependence, or asymmetry of the unknown copula.

\section*{Acknowledgments}

The authors are thankful to the anonymous reviewers and the area editor whose comments helped us to improve the paper significantly.
The authors acknowledge financial support from the ARIS (Slovenian Research and Innovation Agency, research core funding No. P1-0222 and projects J1-70034 and J1-50002).

\bibliographystyle{amsplain}
\bibliography{biblio}

\end{document}